\documentclass[]{amsart}
\usepackage{mathcmd}
\usepackage{graphicx}
\usepackage{mathptmx}
\usepackage{mathtools}
\usepackage{amssymb}
\usepackage{amsmath}
\usepackage{amsfonts}
\usepackage{latexsym}
\usepackage{amsthm}
\usepackage[utf8]{inputenc}

\usepackage[all]{xy}

\usepackage{appendix}

\usepackage[T1]{fontenc}
\usepackage{lmodern}
\usepackage[english]{babel}

\usepackage{blindtext}

\def \tt#1{\text{\rm #1}}

\def \etale {\'{e}tale }

\def \etalee {\'{e}tale}

\def \ab {\tt{ab}}
\def \log {\tt{log}}
\def \Hom {\tt{Hom}}
\def \Aut {\tt{Aut}}

\def \Inn {\tt{Inn}}
\def \Out {\tt{Out}}

\def \Rank {\tt{Rank}}
\def \tf {\tt{tf}}
\def \cycl {\tt{cycl}}
\def \Z {\mathbb{Z}}
\def \G {\mathbb{G}}
\def \Zhat {\hat{\mathbb{Z}}}
\def \Q {\mathbb{Q}}

\def \C {\mathcal{C}}
\def \D {\mathcal{D}}
\def \F {\mathcal{F}}
\def \abt {\text{\rm ab-t}}
\def \sp {\tt{Spec}}
\def \cusp {\text{\rm Cusp}}
\def \id {\tt{Id}}
\def \T {\tt{T}}
\def \tor {\tt{tor}}
\def \abe {\tt{abel}}

\def \Pic {\tt{Pic}}
\def \Det {\tt{Det}}
\def \Cusp {\text{\rm Cusp}}

\def \diag {\tt{diag}}

\def \card {\tt{Card}}

\def \ob {\tt{ob}}
\def \hom {\tt{Hom}}
\def \dom {\tt{dom}}
\def \cod {\tt{cod}}
\def \und {\tt{grph}}

\def \set {\tt{set}}

\def \gal {\tt{Gal}}

\def \base {\tt{Base}}
\def \pr {\tt{pr}}
\def \surjto {\twoheadrightarrow}
\def \injto {\hookrightarrow}
\def \cl {\tt{cpt}}
\def \comm {\mathcal{G}} 
\def \simto {\xrightarrow{\sim}}

\def \sx {\mathbb{S}_{3,X}}
\def \sxx {\mathbb{S}_{4,X}}
\def \sxxx {\mathbb{S}_{5,X}}
\def \vx {\mathbb{V}_{X}}

\usepackage{cleveref}
\crefname{section}{\S}{\S}
\Crefname{section}{\S}{\S}

\crefname{theo}{Theorem}{Theorem}
\Crefname{theo}{Theorem}{Theorem}
\crefname{defi}{Definition}{Definition}
\Crefname{defi}{Definition}{Definition}
\crefname{prop}{Proposition}{Proposition}
\Crefname{prop}{Proposition}{Proposition}

\crefname{lemm}{Lemma}{Lemma}
\Crefname{lemm}{Lemma}{Lemma}

\crefname{coro}{Corollary}{Corollary}
\Crefname{coro}{Corollary}{Corollary}

\crefname{remm}{Remmark}{Remmark}
\Crefname{remm}{Remmark}{Remmark}

\crefname{thma}{Theorem A}{Theorem A}
\Crefname{thma}{Theorem A}{Theorem A}

\theoremstyle{theorem}
\newtheorem{lemm}{Lemma}[section]
\newtheorem{theo}[lemm]{Theorem}
\newtheorem{prop}[lemm]{Proposition}
\newtheorem{coro}[lemm]{Corollary}
\newtheorem*{thma}{Theorem A}
\newtheorem*{thma'}{Theorem A$'$}
\newtheorem*{thmb}{Theorem B}
\newtheorem*{thmc}{Theorem C}
\newtheorem*{thmc'}{Theorem C$'$}

\theoremstyle{remark}
\newtheorem{remm}[lemm]{Remark}

\theoremstyle{definition}
\newtheorem{defi}[lemm]{Definition}

\usepackage{lipsum}

\newcommand{\longtitle}[1]{
	\ifodd\value{page}
	\protect\parbox{0.97\linewidth}{#1}
	\else
	\fi
}

\calclayout

\title[The Absolute Anabelian Geometry of Virtual Curves]{The Absolute Anabelian Geometry of Virtual Curves Arising from Sections of Arithmetic Fundamental Groups of Configuration Spaces}
\author{Zeming Sun}

\begin{document}

\maketitle


\section{Introduction}\label{Introductions}

The objective of this paper is to introduce the notions of \emph{virtual varieties} [cf. \cref{l41}] and \emph{pointed virtual curves} [cf. \cref{l23}] (a special case of the notion of a virtual variety), and to study various anabelian reconstructions related to pointed virtual curves, especially, pointed virtual curves that arise from $2$-configuration spaces [cf. \cref{Notations and Terminologies}, the discussion entitled ``Configuration Spaces'']. Roughly speaking, given a family of pointed curves, the notion of a pointed virtual curve is a generalization of the notion of a fiber of that family over a rational point, where the rational point is replaced by a Galois section of the \etale fundamental group of the base space of the family.



Under suitable assumptions, if the Galois section arises from a rational point $x$, then the virtual fundamental group associated to the pointed virtual curve is naturally isomorphic, up to composition with inner automorphisms, to the \etale fundamental group of the fiber of $x$ [cf. \cref{l01}]. Thus one may regard pointed curves as the special class of pointed virtual curves for which the given section is geometric. In particular, if one considers the family of pointed curves associated to a $2$-configuration space, then one may regard the pointed virtual curve that arises from a section as the ``cuspidalization'' of that section. 

In fact, our main interest lies in the family of pointed curves associated to a $2$-configuration space. Roughly speaking, we shall consider various anabelian reconstructions associated to the abstract virtual fundamental group
\begin{displaymath}
\Pi_{[\pr^{2/1}_{X},s]}
\end{displaymath}
[cf. \cref{l41}; \cref{Notations and Terminologies}, the discussion entitled ``Configuration Spaces''] that arises from the $2$-configuration space of a curve $X$ satisfying certain
conditions and a given Galois section $s:G_{k}\to \Pi_{X}$.


Our main result is as follows:


\begin{thma}\label{thma}
	([cf. Theorem 6.14(iii); Theorem 6.15(vi)(v)(vi)])
	
	Let $k$ be either a number field or a mixed-characteristic local field; $\bar{k}$ an algebraic closure of $k$; $G_{k}$ the Galois group $\gal(\bar{k}/k)$; $X$ a smooth, geometrically connected curve over $k$ of type $(0,r)$, where $r\geq 3$; $s:G_{k}\rightarrow \Pi_{X}$ a section of the natural surjection $\Pi_{X}\surjto G_{k}$. Write $\pr^{2/1}_{X}:X_{2}\rightarrow X$ for the first projection; $\mathbf{F}$ for the category of fields; $\Pi_{s}$ for the virtual fundamental group $\Pi_{[\pr^{2/1}_{X},s]}$ [cf. \cref{l41}; \cref{Notations and Terminologies}, the discussion entitled ``Configuration Spaces''].
	
    (i) One may reconstruct a field 
        \begin{displaymath}
            \base(\Pi_{s})	
	\end{displaymath}
	from the abstract profinite group $\Pi_{s}$ functorially with respect to isomorphisms of topological groups such that the action of $\Pi_{s}$ on $\base(\Pi_{s})$ functorially induced by the conjugation action of $\Pi$ on itself factors through the quotient $\Pi_{s}\surjto G_{k}$. Moreover, there exists a $G_{k}$-equivariant isomorphism $\base(\Pi_{s})\simto \bar{k}$.
	
	(ii) The isomorphism class of the function field $K(X)$ is determined by the isomorphism class of the abstract profinite group $\Pi_{s}$.
	
	(iii) Suppose that $r=3$. Then the isomorphism class of the scheme $X$ is determined by the isomorphism class of the abstract profinite group $\Pi_{s}$.
\end{thma}

In the course of proving Theorem A, we prove various intermediate results,
some of which are of interest independently of Theorem A since they may be proven under \emph{weaker hypotheses} than the hypotheses of Theorem A.

One such intermediate result is the following Theorem B, which may be
regarded as a partial generalization to the case of pointed virtual curves 
of \cite{MZK3}, Lemma 4.5.  The technique of proof of Theorem B is similar to that of \cite{MZK3}, Lemma 4.5.

\begin{thmb}
	([cf. \cref{l24}, \cref{l19}(ii), \cref{l34}, \cref{l39}])
	
	Let $k$ be a field of characteristic zero; $\C$ a pointed virtual curve [cf. \cref{l23}] of type $(g,r)$ with source $X$ and base $k$; $\bar{k}$ an algebraic closure of $k$; $G_{k}$ the Galois group $\gal(\bar{k}/k)$. Assume that $r\geq 1$.
	
	(i) The geometric fundamental group $\Delta_{\C}$, considered as a subgroup of $\Pi_{X}$, is \emph{independent} of the choice of $s$.
	
	(ii) Let $l$ be any prime such that $k$ is $l$-cyclotomically full [cf. \cref{l43}(ii)]. Consider the determinant map $\Det_{\Delta_{\C}}^{l}:\Pi_{\C}\rightarrow (\mathbb{Z}_{l}^{\times})^{\tf}$ [cf. \cref{l43}(iii)]. The homomorphism $(\Det_{\Delta_{\C}}^{l})^{\abt}:(\Pi_{\C})^{\abt}\rightarrow (\mathbb{Z}_{l}^{\times})^{\tf}$ induced on torsion-free quotients factors through $\Pi_{\C}\rightarrow G_{k}$, hence may be considered as a character of $G_{k}$. Then $ w^{l}_{(\cycl_{G_{k}})^{\abt}}((\Det_{\Delta_{\C}})^{\abt})$ [cf. \cref{l43}(i)] exists and is \emph{independent} of the section $s$.
	
	(iii) Suppose that there exists at least one $l$ that satisfies the condition in (ii). Then one may reconstruct the set of conjugacy classes of \emph{cuspidal inertia subgroups} $\Cusp(\C)$ [cf. \cref{l27}] of $\Delta_{\C}$.
\end{thmb}

Also, we may obtain a
\begin{displaymath}
\text{``semi-absolute-anabelian''}\Rightarrow \text{``absolute-anabelian''}
\end{displaymath}
type result:

\begin{thmc}
	([cf. \cref{l38}, \cref{l37}])
	
	Suppose that $k$ is either a number field or a mixed-characteristic local field; $X$ is a smooth, geometrically connected \emph{hyperbolic} curve over $k$. Let $\bar{k}$ be an algebraic closure of $k$; $G_{k}$ the Galois group $\gal(\bar{k}/k)$; $s:G_{k}\rightarrow \Pi_{X}$ a section of the natural surjection $\Pi_{X}\surjto G_{k}$. Then the exact sequence 
	\begin{displaymath}
	1\rightarrow \Delta_{[\pr^{2/1}_{X},s]}\rightarrow \Pi_{[\pr^{2/1}_{X},s]}\rightarrow G_{k}\rightarrow 1
	\end{displaymath}
	may be reconstructed group-theoretically from $\Pi_{[\pr^{2/1}_{X},s]}$. In particular, the group $\Delta_{[\pr^{2/1}_{X},s]}$ is \emph{characteristic} in $\Pi_{[\pr^{2/1}_{X},s]}$.
\end{thmc}

This paper is organized as follows: In \cref{Virtual Curves}, introduce the notion of a virtual curve and discuss some basic properties of virtual curves. In \cref{Weights of Determinant Maps}, we prove Theorem B by considering the Weil pairing associated to the Albanese variety [cf. \cite{MZK3}, Appendix] of the curve under consideration and deduce a generalization of \cite{MZK3}, Lemma 4.5. In \cref{Characterization of Geometric Fundamental Groups}, we discuss various properties of configuration spaces and prove Theorem C by using a ``stricter'' family of open subgroups than the one in \cite{MZK6}, Lemma 1.1.4. In \cref{Automorphisms of 2-Configuration Spaces}, we discuss how to ``decuspidalize'' the diagonal of configuration space so that we may reduce to the case of a fundamental group of a curve of strictly Belyi type and, by applying methods in \cite{MZK2}, Chapter 1, reconstruct various copies of the function field of the curve, thus proving Theorem A.

The notion of a pointed virtual curve may be regarded as a partial generalization of the notion of a \emph{hyperbolic curvoid} in \cite{Hos4}.

\section{Notations and Terminologies}\label{Notations and Terminologies}

~\\
\textbf{Groups and Topologies:}

Throughout this paper, all maps between topological spaces and all morphisms between profinite groups are to be understood as \emph{continuous} without further mention. We shall regard groups not equipped with a topology as topological groups whose topology is the discrete topology.

The notation $1$ will be used to denote the trivial group.

Let $G$ be a profinite group; $\Sigma$ a non-empty set of prime numbers; $n$ an integer. A \emph{$\Sigma$-integer} is defined to be integer whose prime divisors are contained in $\Sigma$. We shall refer to the quotient (whose kernel is easily verified to be a characteristic closed subgroup)

\begin{displaymath}
G \surjto \varprojlim G/H
\end{displaymath}
of $G$, where $H$ ranges over the \emph{normal} open subgroups of $G$ such that $\left[ G:H\right]$ is a $\Sigma$-integer, as the \emph{maximal pro-$\Sigma$ quotient} of $G$, which we denote by $G^{(\Sigma)}$.


Let $G$ be a profinite group. We say that $G$ is \emph{elastic} if every nontrivial topologically finitely generated closed
normal subgroup of an open subgroup of $G$ is itself open.

Denote by $G^{\tf}$ the the maximal torsion-free Hausdorff quotient of an abelian topological group $G$. Denote by $G^{\ab}$ the maximal abelian Hausdorff quotient of a topological group $G$. Write $G^{\abt}\coloneqq(G^{\ab})^{\tf}$. Note that if $G$ is a profinite group (respectively, $G$ is an abelian profinite group), then $G^{\ab}$ (respectively, $G^{\tf}$) admits a naturally induced profinite structure. If $f:G\rightarrow H$ is a morphism of profinite groups, we denote the naturally induced morphism by $f^{\abt}:G^{\abt}\rightarrow H^{\abt}$.

Let $G$ be a topological group; $A$ a topological $G$-module. Then Denote by $A_{G}$ the maximal Hausdorff quotient of $A$ on which $G$ acts trivially; denote by $A^{G}$ the topological submodule of $A$ consisting of the elements fixed by $G$.

Let $n$ be a positive integer. Then denote by $\mathbb{S}_{n}$ the permutation group on $n$ letters.

Let $H$ be a subgroup of finite index of a group $G$. Then denote by $[G:H]$ the index of $H$ in $G$.

Let $G$ be a topological group. Then denote by $\Aut(G)$ the group of automorphisms of $G$; denote by $\Inn(G)$ the group of inner automorphisms of $G$; denote by $\Out(G)$ the group of outer automorphisms of $G$ (i.e., the quotient $\Aut(G)/\Inn(G)$). Suppose further that $G$ is \emph{topologically finitely generated}. Then $G$ admits a topological basis of characteristic open subgroups, which induces a natural topology on $\Aut(G)$ and $\Out(G)$.

Let $S$ be a set. Then denote by $\Aut_{\set}(S)$ the group of automorphisms of $S$; denote by $\card(S)$ the cardinality of $S$.

~\\
\textbf{Schemes and Curves:}

Let $A$ be a ring. Then denote by $\sp(A)$ the spectrum of $A$.

Let $X$, $Y$ be reduced connected schemes. We say that a morphism $f:Y\to X$ of schemes is a \emph{family of curves} if there exist smooth morphisms $g:Z\to X$, $i:Y\to Z$ of schemes, where $g$ is proper and geometrically connected of relative dimension one, and $i$ is an open immersion with dense image, such that $f=g\circ i$, and the complement (equipped with the reduced induced scheme structure) of the image of $Y$ in $Z$ (which may be empty) is finite \etale over $X$.


Let $k$ be a field. We say that a scheme $X$ is a \emph{curve} over $k$ if $X\to \sp(k)$ is a family of curves. 

Let $X$ be an integral scheme. Then denote by $K(X)$ the function field of $X$.

Let $X$ be a smooth curve over a field $k$; $\bar{k}$ an algebraic closure of $k$. Then denote by $X^{\cl}$ the (unique) smooth compactification of $X$. Write $\bar{X}\coloneqq X\times_{k} \bar{k}$, $\bar{X}^{\cl}\coloneqq X^{\cl}\times_{k} \bar{k}$. Write $g(X)$ for the genus of $\bar{X}^{\cl}$; $r(X)$ for the cardinality of the set of points in $\bar{X}^{\cl}\setminus\bar{X}$ (i.e., the set of \emph{geometric points} of $X^{\cl}\setminus X$). We shall refer to the pair $(g(X),r(X))$ as the \emph{type} of $X$. We say that a smooth curve (respectively, a family of curves) is \emph{hyperbolic} if it is of type $(g,r)$ (respectively, each of its fibers is of type $(g,r)$), where $2g+r\geq 3$.

We say that two connected schemes $X$ and $Y$ are \emph{isogenous} if there exists a connected scheme $Z$ such that there exist finite surjective \etale coverings $Z\rightarrow X$, $Z\rightarrow Y$.

We shall say that a curve $X$ is of \emph{strictly Belyi type} if $X$ is defined over a number field (i.e., obtained by base changing a curve over a number field) and isogenous to a smooth curve of type $(0,r)$ where $r\geq 3$.

Denote by $\G_{m,X}$ the multiplicative group scheme over a scheme $X$. If the base field $k$ is fixed, then we often abbreviate $\G_{m,k}$ as $\G_{m}$.

Let $k$ be a field, $A$ a group scheme over $k$. Then denote by $0_{A}$ the $k$-rational point representing the identity.

~\\
\textbf{Log Schemes:}

We refer to \cite{Kat1} for the definition of log schemes, log
\etale morphisms, and log smooth morphisms. We say that a log scheme $X^{\log}$ is \emph{log regular} if $X^{\log}$ satisfies the regularity condition in \cite{MZK4}, Definition 1.1. For a log scheme $X^{\log}$, we shall refer to the maximal open subscheme over which the log structure is trivial as the \emph{interior} of $X^{\log}$ and denote it by $(X^{\log})^{\circ}$.

Let $X$ be a scheme; $D$ a reduced closed subscheme; $X^{\log}$ a log scheme such that $X$ is the underlying scheme of $X^{\log}$. Then we say that $X^{\log}$ is induced by $D$ from $X$ if the log structure on $X^{\log}$ is defined by sub-monoid of $\mathcal{O}_{X}$ that consists of elements that are invertible outside $D$.


~\\
\textbf{Fundamental Groups:}

For a connected scheme $X$ with a geometric point $\bar{x}\rightarrow X$, denote by $\pi_{1}(X,\bar{x})$ the \etale fundamental group of $X$ with basepoint $\bar{x}$.

Throughout the paper, we omit the basepoint since there exists a natural isomorphism between fundamental groups associated to different basepoints which is well-defined up to an inner automorphism. We denote the \'{e}tale fundamental group of $X$ with basepoint omitted as $\Pi_{X}$. If $X$ is the spectrum of a field $k$, we also use the notation $G_{k}$ (i.e., the ``absolute Galois group of $k$'') to denote $\Pi_{X}$. Let $X$ be a scheme that is geometrically connected over a field $k$. Then denote by $\Delta_{X}$ the kernel of the natural surjection $\Pi_{X}\surjto G_{k}$, i.e., the \emph{geometric fundamental group}.

Any morphism $f:X\rightarrow Y$ of connected schemes induces an outer homomorphism $\Pi_{X}\rightarrow \Pi_{Y}$, which we denote by $f_{*}$.

Let $X^{\log}$ be a connected locally Noetherian fs log scheme [cf. \cite{Kat1}, 1.6]. Then denote by $\Pi_{X^{\log}}$ its log fundamental group [cf, \cite{Kat1}, Theorem 10.2, (3)] (again, up to an inner automorphism). Any morphism $f:X^{\log}\rightarrow Y^{\log}$ of connected locally Noetherian fs log schemes induces an outer homomorphism $\Pi_{X^{\log}}\rightarrow \Pi_{Y^{\log}}$, which we denote by $f_{*}$. Note that by \emph{log purity} [cf. \cite{MZK4}, Theorem B], the log fundamental group of a connected regular fs log scheme is isomorphic to the \etale fundamental group of its interior.

Throughout the paper, when we consider various morphisms and constructions of groups that arise from fundamental groups, without further explanation, we assume that those morphisms and constructions are only defined up to an inner automorphism.

~\\
\textbf{Configuration Spaces:}

Let $X$ be a variety over a field $k$. Define the \emph{2-configuration space} $X_{2}$ of $X$ to be the complement in $X\times_{k} X$ of the diagonal. The first (respectively, second) projection $\pr^{1}_{X}$ (respectively, $\pr^{2}_{X}$) from $X\times_{k} X$ to $X$ naturally restricts to $X_{2}$. Denote the restriction by $\pr^{2/1}_{X}$ (respectively, $\pr^{1\backslash 2}_{X}$).

~\\
\textbf{Numbers:}

Denote the set of positive integers by $\mathbb{N}^{+}$.

Denote the profinite completion (as an additive group) of $\mathbb{Z}$ by $\hat{\mathbb{Z}}$. Let $l$ be a prime number. Then denote by $\Z_{l}$ the pro-$l$ completion of $\Z$. The ring structure on $\Z$ naturally induces a ring structure on $\hat{\mathbb{Z}}$ and $\Z_{l}$. Note that $\Z_{l}$ is an integral domain. Denote by $\Q_{l}$ the function field of $\Z_{l}$. Write $\hat{\mathbb{Z}}^{\times}$ (respectively, $\Z_{l}^{\times}$) for the (topological) group of units of $\hat{\mathbb{Z}}$ (respectively, $\Z_{l}$). 

For positive integers $a,b$, denote by $\gcd(a,b)$ their greatest common divisor.

~\\
\textbf{Graphs and Categories:}

Throughout the paper, without further notification, all graphs will be \emph{directed} and allowed to have loops and multiple edges. Let $\mathfrak{G}$ be a graph. Then denote by $V(\mathfrak{G})$ the set of vertices. Denote by $E(\mathfrak{G})$ the set of edges. For any edge $e\in E(\mathfrak{G})$, denote by $h(e)$ (respectively, $t(e)$) the head (respectively, tail) of $e$.


Let $\mathbf{C}$ be a small category. Then denote by $\ob(\mathbf{C})$ (respectively, $\hom(\mathbf{C})$) the set of objects (respectively, morphisms) of $\mathbf{C}$. For any morphism $f \in\hom(\mathbf{C})$, denote by $\dom(f)$ (respectively, $\cod(f)$) the domain (respectively, codomain) of $f$. We shall refer to as a \emph{$\mathbf{C}$-graph} a collection of data $(\mathfrak{G},\mathbf{V},\mathbf{E})$, where $\mathfrak{G}$ is a graph; $\mathbf{V}:V(\mathfrak{G})\rightarrow \ob(\mathbf{C})$, $\mathbf{E}:E(\mathfrak{G})\rightarrow \hom(\mathbf{C})$ are maps of sets such that $\dom(\mathbf{E}(e))=\mathbf{V}(t(e))$, $\cod(\mathbf{E}(e))=\mathbf{V}(h(e))$.

Note that any small category $\mathbf{C}$ determines an \emph{associated graph} $\und(\mathbf{C})$, i.e., by replacing objects by vertexes, morphisms by edges, domains by tails, and codomains by heads. Thus, the notion of a $\mathbf{C}$-graph $(\mathfrak{G},\mathbf{V},\mathbf{E})$ is equivalent to a morphism of graphs $D_{(\mathfrak{G},\mathbf{V},\mathbf{E})}:\mathfrak{G}\to \und(\mathbf{C})$, which we refer to as the \emph{defining morphism} of $(\mathfrak{G},\mathbf{V},\mathbf{E})$.

Let $\mathbf{C}$ be a small category; $(\mathfrak{G},\mathbf{V},\mathbf{E})$, $(\mathfrak{G}',\mathbf{V}',\mathbf{E}')$ $\mathbf{C}$-graphs. Then we define a morphism
\begin{displaymath}
(f, \tau):(\mathfrak{G},\mathbf{V},\mathbf{E})\to (\mathfrak{G}',\mathbf{V}',\mathbf{E}')
\end{displaymath}
to be a collection of data $(f, \tau)$, where $f:\mathfrak{G}\rightarrow \mathfrak{G}'$ is a morphism of graphs, and $\tau:V(\mathfrak{G})\rightarrow \hom(\mathbf{C})$ is a morphism of sets such that for any vertex $v\in V(\mathfrak{G})$, $t(\tau(v))=\mathbf{V}(v)$, $h(\tau(v))=\mathbf{V}'(f(v))$, and for any edge $e\in E(\mathfrak{G})$, $\tau(h(e))\circ \mathbf{E}(e)=\mathbf{E'}(f(e))\circ \tau(t(e))$. Thus, a morphism of $\mathbf{C}$-graphs is a sort of ``natural transformation''.

Let $\mathbf{C}$ be a category; $A\in \ob(\mathbf{C})$ an object; $\mathfrak{G}$ a graph. Then define the \emph{constant $\mathbf{C}$-graph $A_{\mathfrak{G}}$ valued in $A$ over $\mathfrak{G}$} to be the $\mathbf{C}$-graph $(\mathfrak{G},\mathbf{V},\mathbf{E})$, where $\mathbf{V}(v)=A$ for any vertex $v\in V(\mathfrak{G})$, and $\mathbf{E}(e)=\id_{A}$ for any edge $e\in E(\mathfrak{G})$.



\section{Virtual Curves}\label{Virtual Curves}

In anabelian geometry, one important problem is the \emph{cuspidalization problem}. Roughly speaking, the cuspidalization problem concerns the reconstruction of the \etale fundamental group $\Pi_{X\setminus x}$ from the abstract topological group $\Pi_{X}$, where $X$ is a (typically affine) geometrically connected smooth curve over a field $k$ and $x$ is a $k$-rational closed point.

Of course, if one is not given any data about the closed point $x$, one cannot expect to reconstruct $\Pi_{X\setminus x}$. In general, there are at least two ways to resolve this issue. One way is to replace the datum of an abstract group $\Pi_{X}$ by the datum of a section $s:G_{k}\to \Pi_{X}$ of the natural surjection $\Pi_{X}\surjto G_{k}$, i.e., which corresponds to the $k$-rational point $x$. The other way is to reconstruct the \emph{set of all possible cuspidalizations} (with some restrictions) rather than to reconstruct a specific one.

One approach to the cuspidalization problem is the method of \emph{configuration spaces}. Roughly speaking, if one may reconstruct the morphism $(\pr^{2/1}_{X})_{*}:\Pi_{X_{2}}\to \Pi_{X}$ [cf. \cref{Notations and Terminologies}, the discussion entitled ``Configuration Spaces''] from $\Pi_{X}$ (where we note that this is in general not possible without additional
assumptions), then one may reconstruct $\Pi_{X\setminus x}$ as the preimage of $s(G_{k})$ by $(\pr^{2/1}_{X})_{*}$, where $s:G_{k}\to \Pi_{X}$ is the section of the natural surjection $\Pi_{X}\surjto G_{k}$ that arises from the $k$-rational point $x$ [cf. \cref{l01}].

One problem with the approach discussed in the preceding paragraph
is that it depends on the auxiliary datum of the set of sections of
$\Pi_{X}\surjto G_{k}$ that arise from $k$-rational points. This state of affairs leads naturally to the idea that one should generalize the notion of cuspidalization so that it applies to arbitrary sections of $\Pi_{X}\surjto G_{k}$.
This point of view is realized in the following definition of the notion of a \emph{virtual variety}.


\begin{defi}\label{l41}\quad
	
	(i) A \emph{virtual variety} $\mathcal{C}$ is a collection of data $(X,Y,f,k,\bar{k},t,s)$, where
	\begin{itemize}
		\item $k$ is a field of \emph{characteristic zero}, $\bar{k}$ is an algebraic closure of $k$;
		\item $X$, $Y$ are connected separated \emph{normal} schemes of finite type over $k$ such that the natural morphism $t_{*}:\Pi_{Y}\to \gal(\bar{k}/{k})$ is surjective, where $t$ is the natural morphism $Y\to \sp(k)$;
		\item $f:X\rightarrow Y$ is a smooth $k$-morphism of $k$-schemes such that the naturally induced morphism $f_{*}:\Pi_{X}\rightarrow \Pi_{Y}$ is surjective;
		\item $s:\gal(\bar{k}/{k})\rightarrow \Pi_{Y}$ is a section of $t_{*}$.
	\end{itemize}
	We refer to $k$ as the \emph{base} of $\mathcal{C}$, $X$ as the \emph{source} of $\mathcal{C}$, $Y$ as the \emph{target} of $\mathcal{C}$, $f$ as the \emph{defining morphism} of $\mathcal{C}$, $t$ as the \emph{structure morphism} of $\mathcal{C}$, and $s$ as the \emph{section} of $\mathcal{C}$.
	
	(ii) Define the pull-back of the diagram
	\begin{displaymath}
	\xymatrix{
		\Pi_{X}\ar[r]^{f_{*}} & \Pi_{Y} \\
		& \gal(\bar{k}/{k}) \ar[u]^{s}
	}
	\end{displaymath}
	to be the \emph{virtual fundamental group} of the virtual variety $\mathcal{C}$, and denote this group by $\Pi_{\mathcal{C}}$. There exists a natural surjective morphism $\Pi_{\mathcal{C}}\surjto \gal(\bar{k}/{k})$. We shall refer to the kernel $\Delta_{\mathcal{C}}$ of this morphism as the \emph{geometric virtual fundamental group}.
	
	(iii) For simplicity, if there is no ambiguity in the base $k$, we often abbreviate the collection of data $(X,Y,f,k,\bar{k},t,s)$ as $[f,s]$.
\end{defi}


The following corollary follows immediately from the definitions.
\begin{coro}\label{l24}
	Let $\mathcal{C}:(X,Y,f,k,\bar{k},t,s)$ be a virtual variety. Then there exist natural exact sequences
	\begin{displaymath}
	1\rightarrow \Delta_{\mathcal{C}} \rightarrow \Pi_{\mathcal{C}} \rightarrow \gal(\bar{k}/{k}) \rightarrow 1
	\end{displaymath}
	\begin{displaymath}
	1\rightarrow \Delta_{\mathcal{C}} \rightarrow \Pi_{X} \rightarrow \Pi_{Y} \rightarrow 1.
	\end{displaymath}
	Moreover, $\Delta_{\mathcal{C}}$, considered as a subgroup of $\Pi_{X}$, is independent of the choice of $s$.	
\end{coro}

\begin{defi}\label{l45}
	Let $\C':(X',Y',f',k',\bar{k}',t',s')$, $\C:(X,Y,f,k,\bar{k},t,s)$ be virtual varieties. Then a \emph{morphism $\mathcal{F}:\C'\rightarrow \C$ of virtual varieties} is an ordered triple $(h,i,j)$, where
	\begin{itemize}
		\item $j:\sp(k')\rightarrow \sp(k)$ is a finite morphism;
		\item $h:X'\rightarrow X$ and $i:Y'\rightarrow Y$ are $k$-morphisms between the respective sources and targets
	\end{itemize}
	such that the diagram
	\begin{displaymath}
	\xymatrix{
		X'\ar[r]^{f'}\ar[d]^{h} & Y'\ar[d]^{i} \\
		X\ar[r]^{f} & Y 
	}
	\end{displaymath}
	and the induced diagram
	\begin{displaymath}
	\xymatrix{
		\Pi_{X'} \ar[r]^{f_{*}'}\ar[d]^{h_{*}} & \Pi_{Y'}\ar[d]^{i_{*}} & \gal(\bar{k}'/{k}')\ar[l]^{s'}\ar[d]^{j_{*}}\\
		\Pi_{X} \ar[r]^{f_{*}} & \Pi_{Y} & \gal(\bar{k}/{k})\ar[l]^{s}
	}
	\end{displaymath}
	commute. We say that $\F$ is \emph{\'etale} if $h$, $i$ are \etalee. We say that $\F$ is an \emph{\etale covering} if $h$, $i$ are \etale coverings.
\end{defi}

The following corollary follows immediately from the definitions.
\begin{coro}\label{l20}\quad
	
	(i) Let $\F:\C'\rightarrow \C$ be a morphism of virtual varieties. Then $\F$ induces compatible morphisms $\pi_{\F}:\Pi_{\C'}\rightarrow \Pi_{\C}$, $\delta_{\F}:\Delta_{\C'}\rightarrow \Delta_{\C}$ of topological groups. If $\F$ is an \etale covering, then $\pi_{\F}$, $\delta_{\F}$ are inclusions of open subgroups. 
	
	(ii) Conversely, let $\C$ be a virtual variety. Then for any open subgroup $H\subset \Pi_{\C}$, there exists an \etale covering $\C'\rightarrow \C$ such that the image of $\Pi_{\C'}$ in $\Pi_{\C}$ is $H$.
\end{coro}

Note that the \etale covering $\C'\rightarrow \C$ of \cref{l20}(ii) is \emph{not necessarily unique up to isomorphism}.

In this paper, our main interest lies in the following special class
of virtual varieties.

\begin{defi}\label{l23}
	A virtual variety $(X,Y,f,k,\bar{k},t,s)$ is referred to as a \emph{pointed virtual curve of type $(g,r)$} if $f$ is a family of hyperbolic curves of type $(g,r)$.
\end{defi}

\begin{coro}\label{l21}
	Any \etale covering of a pointed virtual curve is a pointed virtual curve (possibly of a different type).
\end{coro}

\begin{proof}
	This follows immediately from \cite{Hos3}, Proposition 2.3.
	
\end{proof}

The following proposition shows that the language of fundamental groups
of virtual varieties behaves in the expected way (at least in the case of 
pointed virtual curves), in the sense that there exists a natural isomorphism from the geometric fundamental group of a fiber (of a rational point) to the geometric virtual fundamental group.

\begin{prop}\label{l01}
	Let $\C:(X,Y,f,k,\bar{k},t,s)$ be a pointed virtual curve; $\bar{y}$ a geometric point of $Y$. Then the natural group morphism
	\begin{displaymath}
	\Pi_{X_{\bar{y}}}\rightarrow \Delta_{\mathcal{C}}
	\end{displaymath}
	is an isomorphism.
\end{prop}

\begin{proof}
	This is a special case of \cite{Hos3}, Proposition 2.4(i).
\end{proof}

The following corollary follows immediately from \cref{l01}.

\begin{coro}\label{l02}
	Let $\C:(X,Y,f,k,\bar{k},t,s)$ be a virtual variety; $y$ a $k$-rational point of $Y$; $\bar{y}$ a geometric point over $y$. Assume that $s$ arises (up to an inner automorphism) from the point $y$. Then there exists a natural commutative diagram
	\begin{displaymath}
	\xymatrix{
		\Pi_{X_{\bar{y}}}\ar[r]\ar[d] & \Pi_{X_{y}}\ar[r]\ar[d] & \gal(\bar{k}/k)\ar@{=}[d] \\
		\Delta_{\C}\ar[r] & \Pi_{\C}\ar[r] & \gal(\bar{k}/k).	
	}
	\end{displaymath}
	If we assume further that $\C$ is a pointed virtual curve, then the vertical arrows are isomorphisms.
\end{coro}


\section{Weights of Determinant Maps}\label{Weights of Determinant Maps}

In general, it is difficult to try to recover any sort of scheme structure from the virtual fundamental group. Nevertheless we note that if we work in the situation of \cref{l01}, then the geometric virtual fundamental group is isomorphic to the geometric \etale fundamental group of a certain curve, and it is natural to expect that it should be possible to reconstruct data corresponding to the \emph{cuspidal inertia subgroups} [cf. \cref{l27}]. This makes sense since in the case where the section arises from a rational point [cf. \cref{l02}], the set of conjugacy classes of cuspidal inertia subgroups, which arises from the set of conjugacy classes of decomposition groups determined by the cusps, is \emph{independent} of the (geometric) section chosen [cf. \cref{l39}]. Therefore, we may expect that even if we omit the assumption that the section arises from a rational point, it should still be possible to reconstruct data corresponding to the cuspidal inertia subgroups.

To achieve this goal, we try to mimic the technique developed in \cite{MZK3}, \S 4, especially Lemma 4.5. First, we recall certain definitions and properties of Albanese varieties.

\begin{defi}\quad
	[cf. \cite{MZK3}, Definition A.1]
	
	Let $k$ be a field of characteristic zero; $\comm$ a class of commutative group schemes of finite type over $k$. Suppose that $(X,x)$ is a pair where $X$ is a \emph{geometrically integral} variety over $k$; $x$ is a $k$-rational point of $X$. Then we shall refer to a morphism $f:(X,x)\rightarrow (A,0_{A})$ of pointed $k$-varieties (i.e., ``pointed variety over $k$'' in the sense of \cite{MZK3}, Definition A.1) as a $\comm$-\emph{Albanese morphism} if $A\in\comm$, and for any morphism $f':(X,x)\rightarrow (A',0_{A'})$ of pointed $k$-varieties, where $A'\in\comm$, there exists a unique morphism of group schemes $g:A\rightarrow A'$ such that $f'=g\circ f$. We refer to $A$ as the \emph{$\comm$-Albanese variety} of $X$ (which in fact depends on the choice $x$ of a $k$-rational point of $X$). If $\comm$ is the class of semi-abelian varieties [cf. \cite{FC}, Chapter 1, Definition 2.3], we often abbreviate the term ``$\comm$-Albanese'' as ``Albanese''.
\end{defi}


\begin{prop}\label{l05}
	Let $k$ be a field of characteristic zero; $\bar{k}$ an algebraic closure of $k$; $X$ a \emph{proper smooth} curve over $k$, $x$ a $k$-rational point. Then:
	
	(i) There exists an Albanese morphism
	\begin{displaymath}
	(X,x)\rightarrow (A,0_{A})
	\end{displaymath}
	(where $A$ is actually an \emph{abelian} variety), which induces an isomorphism 
	\begin{displaymath}
	\Pi_{X_{\bar{k}}}^{\ab}\to \Pi_{A_{\bar{k}}}
	\end{displaymath}
	that is compatible with the respective actions by $G_{k}$. Moreover, $A$ is isomorphic to $(\Pic^{0}_{X/k})^{\vee}$, the dual of the connected component of the Picard scheme of $X$ that contains the identity (which is an abelian variety).
	
	(ii) Let $D\subset X$ be a finite set of closed points (which naturally determines a divisor with normal crossings) such that $x\notin D$; $Y$ the complement of $D$ in $X$; $M$ the free $\Z$-module of divisors supported in $D$ (i.e., the free $\Z$-module generated by the points in $D$); $P$ the submodule of $M$ of divisors with degree zero. Regard $x$ as a point of $Y$. Then there exists an Albanese morphism
	\begin{displaymath}
	(Y,x)\rightarrow (A',0_{A'}),
	\end{displaymath}
	which induces an isomorphism
	\begin{displaymath}
	(\Pi_{Y_{\bar{k}}})^{\ab}\xrightarrow{\sim} \Pi_{A'_{\bar{k}}}
	\end{displaymath}
	that is compatible with the respective actions by $G_{k}$. Moreover, if each closed point of $D$ is $k$-rational, then there exists a naturally induced exact sequence (of group schemes over $k$)
	\begin{displaymath}
	1\rightarrow B\rightarrow A'\rightarrow A\rightarrow 1,
	\end{displaymath}
	where $B$ is naturally isomorphic over $k$ to the \emph{torus} with \emph{character group} $P$.
\end{prop}
\begin{proof}
	The assertions of \cref{l05} follow immediately from \cite{MZK3}, Proposition A.6, and \cite{MZK3}, Proposition A.8, together with the fact that $H^{2}(X,\mathcal{O}_X)=0$ (which implies that the Picard scheme of $X$ is \emph{smooth}), and the fact that the abelianizations $(\Pi_{X_{\bar{k}}})^{\ab}$, $(\Pi_{Y_{\bar{k}}})^{\ab}$ of the geometric fundamental groups are torsion-free [cf. \cite{MZK1}, Remark 1.2.2].
\end{proof}


\begin{lemm}\label{l04}
	Let $\mathcal{C}:(X,Y,f,k,\bar{k},t,s)$ be a pointed virtual curve of type $(g,r)$ with \emph{integral} target $Y$ and base $k$ of characteristic zero; $\zeta$ the generic point of $Y$; $\bar{\zeta}$ a geometric point over $\zeta$; $K$ the function field of $Y$. Then
	
	(i) $X_{\zeta}\rightarrow \zeta$ is a smooth curve.
	
	(ii) The natural commutative diagram
	\begin{displaymath}
		\xymatrix{
			\Pi_{X_{\bar{\zeta}}}\ar[d]\ar[r] & \Pi_{X_{\zeta}}\ar[d]\ar[r] & \gal(\bar{\zeta}/\zeta)\ar[d]\\
			\Delta_{\mathcal{C}}\ar[r] & \Pi_{X}\ar[r] & \Pi_{Y}
		}
	\end{displaymath}
	[cf. \cref{l24}] induces an (outer) isomorphism
	\begin{displaymath}
		\Pi_{X_{\bar{\zeta}}}\simto \Delta_{\mathcal{C}}.
	\end{displaymath}
\end{lemm}

\begin{proof}
	(i) follows immediately from the definitions. 
	(ii) is a special case of \cref{l01}.
\end{proof}

	

	
	

\begin{defi}
	Let $A$ be a commutative group scheme over a field $k$, $\bar{k}$ a separable closure of $k$. Then: 
	
	(i) Denote by $_{n}A(\bar{k})$ the group of $n$-torsion elements in $A(\bar{k})$. 
	
	(ii) Define the \emph{Tate module} $\T(A_{\bar{k}})$ of $A_{\bar{k}}$ to be the inverse limit topological group
	\begin{displaymath}
	\varprojlim_{n\in \mathbb{N}^{+}} {}_{n}A(\bar{k}),
	\end{displaymath}
	where the transition morphisms are given by
	\begin{displaymath}
	_{mn}A(\bar{k})\rightarrow {}_{n}A(\bar{k}):x\rightarrow mx.
	\end{displaymath}
	
	(iii) Denote by $\T_{\Z_{l}}(A_{\bar{k}})$ the maximal pro-$l$ quotient of $\T(A_{\bar{k}})$. Write $\T_{\Q_{l}}(A_{\bar{k}})\coloneqq\T_{l}(A_{\bar{k}})\otimes_{\Z_{l}}\Q_{l}$.
	
	(iv) Let $n$ be a positive integer. Then denote by $\mu_{n}(\bar{k})$ the group of roots of unity of order $n$ in $\bar{k}$. Note that $\mu_{n}(\bar{k})$ is naturally isomorphic to $_{n}(\G_{m,\bar{k}})$. Let $l$ be any prime number. Then denote by $\mu(\bar{k})$ (respectively, $\mu_{\Z_{l}}(\bar{k})$, $\mu_{\Q_{l}}(\bar{k})$) the Tate module $\T(\G_{m,\bar{k}})$ (respectively, $\T_{\Z_{l}}(\G_{m,\bar{k}})$, $\T_{\Q_{l}}(\G_{m,\bar{k}})$).
	
	(v) Let $n$ be a positive integer. Then denote by $n_{A}:A\to A$ the map given by multiplication by $n$.
	
	(vi) Let $x$ be a $k$-rational point of $A$. Then denote by $t_{x}$ the \emph{translation} by $x$.
\end{defi}

\begin{prop}\label{l06}
	
	Let $A$ be an abelian variety over a field $k$ of characteristic zero; $\bar{k}$ an algebraic closure of $k$; $A^{\vee}$ the dual abelian variety [cf. \cite{Mumf}, Chapter III, \S13, Theorem]; $n$ a positive integer. Then there exists a polarization (i.e., an isogeny)
	\begin{displaymath}
	\theta:A\to A^{\vee}.
	\end{displaymath}
	Each such polarization induces a natural pairing (i.e., a $\Z$-bilinear map)
	\begin{displaymath}
	\theta_{n}: {}_{n}A_{\bar{k}}\times {}_{n}A_{\bar{k}} \rightarrow \mu_{n}(\bar{k})
	\end{displaymath}
	of abelian groups with $G_{k}$-action.
	Let $l$ be any prime number. Then the preceding pairing induces a pairing (i.e., a $\Z_{l}$-bilinear map)
	\begin{displaymath}
	\theta_{\Z_{l}}: \T_{\Z_{l}}(A_{\bar{k}})\times \T_{\Z_{l}}(A_{\bar{k}}) \rightarrow \mu_{\Z_{l}}(\bar{k})
	\end{displaymath}
	of topological abelian groups, which induces a \emph{non-degenerate} bilinear form
	\begin{displaymath}
	\theta_{\Q_{l}}: \T_{\Q_{l}}(A_{\bar{k}})\times \T_{\Q_{l}}(A_{\bar{k}}) \rightarrow \mu_{\Q_{l}}(\bar{k})
	\end{displaymath}
	of $\Q_{l}$-vector spaces with $G_{k}$-action.
	We refer to the pairing $\theta_{\Z_{l}}$ as the pro-${l}$ \emph{Weil pairing} associated to $\theta$.
\end{prop}
\begin{proof}
	The existence of a polarization follows from the fact that any abelian variety admits an ample invertible sheaf [cf. \cite{Milne}, Chapter V, Corollary 7.2; \cite{Milne}, Chapter V, \S13]. The remaining assertions concerning pairings associated to $\theta$ follow immediately from the discussion of \cite{Milne}, \S 16.
	
\end{proof}


	
	
	


	

\begin{prop}\label{l14}
	
	Let $k$ be a field of characteristic zero; $\bar{k}$ an algebraic closure of $k$; $A$ a semi-abelian variety over $k$. Then there exists a natural isomorphism
	\begin{displaymath}
		\Pi_{A_{\bar{k}}}\simto \T(A_{\bar{k}})
	\end{displaymath}
	of profinite abelian groups with $G_{k}$-action.
\end{prop}
\begin{proof}
	
	Without loss of generality, we may assume that $k=\bar{k}$. First, we show that there exists a natural surjection
	\begin{displaymath}
	f:\Pi_{A}\surjto \T(A).
	\end{displaymath}
	For any $n\in \mathbb{N}^{+}$, since $n$ is invertible in $k$, $n_{A}$ is a Galois \etale covering [cf. \cite{SGA3}, Expos\'e $\tt{VI}_{\tt{B}}$, Proposition 1.3, Corollaire 1.4.1]. Hence there exists a surjection
	\begin{displaymath}
	\Pi_{A}\surjto {}_{n}A.
	\end{displaymath}
	Such surjections are compatible with passage to the inverse limit with respect to $n$, hence induce a surjection
	\begin{displaymath}
	\Pi_{A}\surjto \T(A).
	\end{displaymath}
	To show that $f$ is injective, it suffices to show that for any finite Galois \etale covering $g:A'\to A$, there exists an $n\in\mathbb{N}^{+}$ and an \etale covering $h:A\to A'$ such that $g\circ h=n_{A}$. 
	By \cref{l08} below, there exists a \emph{commutative} algebraic group structure on $A'$ such that $g$ is a homomorphism of algebraic groups. Let $n$ be the degree of $g$. Then the kernel of $g$ is contained in the kernel of $n_{A'}$. Therefore, there exists a morphism $h:A\to A'$ such that $h\circ g=n_{A'}$, so $g\circ h\circ g=g\circ n_{A'}=n_{A}\circ g$. Since $g$ is dominant, this implies that $g\circ h=n_{A}$, as desired.
\end{proof}
\begin{remm}
	Note that it follows from the proof of \cref{l14} that, in the notation of this proof, $A'$ is semi-abelian.
\end{remm}

\begin{lemm}\label{l08}
	Let $k$ be an \emph{algebraically closed} field of characteristic zero; $A$ a semi-abelian variety over $k$; $f^{B}_{A}:B\rightarrow A$ a Galois \etale covering of $A$. Then $B$ admits a commutative group scheme structure such that $f^{B}_{A}$ is a homomorphism of algebraic groups.
\end{lemm}
\begin{proof}
	Recall that the group structure on $A$ consists of three data $e_{A}$, $\otimes_{A}$, $(-)^{-1}_{A}$, where $e_{A}:\sp(k)\to A$ is the unit, $\otimes_{A}:A\times_{k} A\to A$ is the multiplication morphism, and $(-)^{-1}_{A}:A\to A$ is the inverse morphism.
	
	First, we prove that there is a natural isomorphism $\Pi_{A\times_{k}A}\simto\Pi_{A}\times \Pi_{A}$. Note that this is not a direct consequence of [\cite{SGA1}, Expos\'e X, Corollaire 1.7] since $A$ is not necessarily proper.
	
	By the definition of the notion of a semi-abelian variety, $A$ may be written as an extension
	\begin{displaymath}
		1\to T\to A\to X\to 1,
	\end{displaymath}
	where $T$ is a torus, and $X$ is an abelian variety. Similarly, $A\times_{k}A$ may be written as an extension
	\begin{displaymath}
	1\to T\times_{k}T\to A\times_{k}A\to X\times_{k}X\to 1,
	\end{displaymath}
	and admits a natural group scheme structure. Consider the commutative diagram of profinite groups
	\begin{displaymath}
		\xymatrix{
			1\ar[r] &\Pi_{T\times_{k}T}\ar[r]\ar[d] & \Pi_{A\times_{k}A}\ar[r]\ar[d] & \Pi_{X\times_{k}X}\ar[d]\ar[r] & 1\\
			1\ar[r] &\Pi_{T}\times \Pi_{T}\ar[r] & \Pi_{A}\times \Pi_{A}\ar[r] & \Pi_{X}\times \Pi_{X}\ar[r] & 1,
		}
	\end{displaymath}
	where the vertical arrows are the natural homomorphisms. By \cref{l09} below, the horizontal sequences are exact. Since $T$ is a split torus, the left-hand vertical arrow is an isomorphism. Since $X$ is proper, the right-hand vertical arrow is an isomorphism [cf. \cite{SGA1}, Expos\'e X, Corollaire 1.7]. Therefore, the middle vertical arrow determines a natural isomorphism $\Pi_{A\times_{k}A}\simto\Pi_{A}\times \Pi_{A}$. A similar argument shows that we have a natural isomorphism $\Pi_{A\times_{k}A\times_{k}A}\simto\Pi_{A}\times \Pi_{A}\times \Pi_{A}$.
	
	Let $f^{B}_{A}:B\to A$ be a \emph{Galois} \etale covering. Fix a point $e_{B}:\sp(k)\to B$ that lies in the preimage of $e_{A}$.
	
	Denote by $g_{A}:A\times_{k} A \simto A\times_{k} A$ the 
	isomorphism determined by mapping $(a,b)\mapsto (a,a+b)$. We claim that there exists an isomorphism $g_{B}:B\times_{k} B \simto B\times_{k} B$ such that the diagram 
	\begin{displaymath}
		\xymatrix{
			B\times_{k} B\ar^{g_{B}}[r]\ar^{f^{B}_{A}\times_{k}f^{B}_{A}}[d] & B\times_{k} B\ar^{f^{B}_{A}\times_{k}f^{B}_{A}}[d]\\
			A\times_{k} A\ar^{g_{A}}[r] & A\times_{k} A
		}
	\end{displaymath}
	commutes. Indeed, consider the cartesian diagram
	\begin{displaymath}
		\xymatrix{
			B\times_{k} B\ar^{h_{1}}[r]\ar^{h_{2}}[d] & B\times_{k} B\ar^{f^{B}_{A}\times_{k}f^{B}_{A}}[d]\\
			A\times_{k} A\ar^{g_{A}}[r] & A\times_{k} A
		}
	\end{displaymath}
	determined by $g_{A}$ and $f^{B}_{A}\times_{k}f^{B}_{A}$.
	Next, observe that there exist isomorphisms of the restrictions to $A\times_k e_A$ and $e_A\times_k A$ of the finite \'etale coverings of $A\times_k A$ determined by $f^B_A\times_k f^B_A$ and $h_2$ such that the restrictions of these isomorphisms coincide after further restriction to $e_A\times_k e_A$.  Thus, it follows formally from the existence of the natural isomorphism $\Pi_{A\times_k A}\simto \Pi_A\times \Pi_A$ proven above that the finite \'etale coverings of $A\times_k A$ determined by $f^B_A\times_k f^B_A$ and $h_2$ are isomorphic, hence that (for a suitable choice of $h_2$) one may take $g_B$ to be $h_1$.  Finally, by composing with a suitable deck transformation  of the finite \'etale covering of $A\times_k A$ determined by $f^B_A\times_k f^B_A$, we may assume that $g_B(e_B\times_k e_B)=e_B\times_k e_B$.

	
	By construction, $\pr^{1}_{B}\circ g_{B}$ lies over $\pr^{1}_{A}$ (with respect to the projections $f^{B}_{A}\times_{k}f^{B}_{A}$ and $f^{B}_{A}$). Since $g_B(e_B\times_k e_B)=e_B\times_k e_B$, we thus conclude that $\pr^{1}_{B}\circ g_{B}=\pr^{1}_{B}$. Therefore, since $g_{B}$ is an isomorphism, $g_{B}^{-1}(B\times e_B)$ is a graph, hence defines a morphism $(-)^{-1}_{B}:B\to B$. One sees immediately that $(-)^{-1}_{B}(e_{B}) =e_{B}$.

	
	Write $\otimes_{B}\coloneqq \pr^{2}_{B}\circ g_{B}$. Denote by $\tau$ the natural automorphism of $B\times_{k} B$ obtained by switching the two factors. Since $A$ is a commutative group scheme, both $\otimes_{B}$ and $\otimes_{B}\circ \tau$ lie over $\otimes_{A}$ (with respect to the projections $f^{B}_{A}\times_{k}f^{B}_{A}$ and $f^{B}_{A}$). Since $g_B(e_B\times_k e_B)=e_B\times_k e_B$, we thus conclude that $\otimes_{B}=\otimes_{B}\circ \tau$. Therefore, the binary operation determined by $\otimes_{B}$ is commutative. Similar arguments show that the binary operation determined by $\otimes_{B}$ is associative, and that $e_{B}$ is a unit. Therefore, the data $e_{B}$, $\otimes_{B}$, $(-)^{-1}_{B}$ defines a commutative group scheme structure on $B$, which is compatible with $f^{B}_{A}$.

\end{proof}

\begin{lemm}\label{l09}
	Let $A$ be a semi-abelian scheme over an algebraically closed field $k$ of characteristic zero. Let 
	\begin{displaymath}
	1\to T\to A\to X\to 1
	\end{displaymath}
	be an extension of group schemes [cf. \cite{SGA3}, Expos\'e $\text{VI}_{\text{A}}$, Th\'eor\`eme 5.4] such that $T$ is a torus. Then the sequence
	\begin{displaymath}
		1\to \Pi_{T}\to \Pi_{A}\to \Pi_{X}\to 1
	\end{displaymath}
	is exact.
\end{lemm}
\begin{proof}
	Since $k$ is algebraically closed, $T$ is a split torus. Thus, by induction, we reduce to the case where $T\simeq \G_{m}$.
	Since $A\to X$ is a $\G_{m}$-torsor in the \'etale topology, the natural action of $\G_{m}$ on $\mathbf{P}^{1}_{k}$ determines a natural partial compactification of $A\to X$ to a $\mathbf{P}^{1}_{k}$-bundle $\bar{A} \to X$ in the \'etale topology such that the complement $\bar{A}\setminus A$ is a relative divisor with normal crossings over $X$, hence determines a log structure on $\bar{A}$.  Thus, if we denote the resulting log scheme by $\bar{A}^\log$, then we obtain a log smooth morphism $\bar{A}^\log \to X^\log$, where we equip $X$ with the trivial log structure. By [\cite{Hos}, \S3, Theorem 2], we obtain an exact sequence
	\begin{displaymath}
		\Pi_{T}\to \Pi_{A}\to \Pi_{X}\to 1.
	\end{displaymath}
	Finally, since for $n\in\mathbb{N}^{+}$, $n_A$ is a Galois \'etale covering [cf. \cite{SGA3}, Expos\'e $\tt{VI}_{\tt{B}}$, Proposition 1.3, Corollaire 1.4.1] that restricts to $n_T$ on $T$, we conclude that the natural homomorphism $\Pi_T\to\Pi_A$ is injective. This completes the proof of \cref{l09}.
\end{proof}

By \cref{l14}, the \etale fundamental group of a semi-abelian variety may be naturally identified with its Tate module. Next, we review the language of \emph{cyclotomic characters} and \emph{weights}.

\begin{defi}\label{l43}\quad
	
	(i) Let $G$ be a profinite group; $l$ a prime number; $\chi_{G}$, $f$ two torsion-free characters (i.e., two continuous homomorphisms $G\rightarrow (\mathbb{Z}_{l}^{\times})^{\tf}$ of profinite groups) of $G$. Then we say that $f$ is of \emph{weight} $w$ relative to $\chi_{G}$ if there exist $a,b\in \mathbb{Z}$, $b\neq 0$ , such that $w=\frac{2a}{b}$, and $(f^{\abt})^{b}=((\chi_{G})^{\abt})^{a}$. (Here, the ``$2$'' appears in order to ensure that the  weights of the characters
	that we are interested in are integers.  This convention is consistent with
	the usual conventions concerning weights that are applied in the context of the Riemann Hypothesis over finite fields.) In general, $f$ might not be of weight $w$ for any rational number $w$.
	On the other hand, if the image of $\chi_G$ is {\it open}, then there exists
	at most one rational number $w$, which we shall denote $w_{\chi_G}(f)$
	when it exists,  such that $f$ is of weight $w$ relative to $\chi_G$.
	
	(ii) Suppose further that $G$ is an \etale fundamental group of a scheme $X$ such that $l$ is invertible in $X$ (respectively, the virtual fundamental group of a pointed virtual curve $\mathcal{C}:(X,Y,f,k,\bar{k},t,s)$ such that $l$ is invertible in $k$). Denote by $\cycl_{G}^{l}:G\rightarrow (\Z_{l}^{\times})^{\tf}$ the (torsion-free) \emph{$l$-adic cyclotomic character} of $G$ (respectively, the composite of the \emph{$l$-adic cyclotomic character} on $k$ with the surjection $G=\Pi_{\mathcal{C}}\surjto G_{k}$). We shall say that $G$, or, alternatively, $X$ (respectively, $\mathcal{C}$), is \emph{$l$-cyclotomically full} if the image of $\cycl_{G}^{l}$ is open. Thus, if $G$ is $l$-cyclotomically full, then for any character $f: G\rightarrow (\mathbb{Z}_{l}^{\times})^{\tf}$ that is of weight $w$ relative to $\cycl_{G}^{l}$ for some rational number $w$, the rational number $w_{\cycl_{G}^{l}}(f)$ is well-defined. 
	
	(iii) Let $H$ be a \emph{topologically finitely generated} profinite group. Suppose that we are given a \emph{continuous} homomorphism $f:G\rightarrow \Aut(H)$. Then $f$ induces a natural action of $G$ on the finitely generated $\Z_{l}$-module $H^{\ab}\otimes_{\hat{\Z}}\Z_{l}$. Thus, $f$ induces a continuous determinate homomorphism $G\rightarrow \Z_{l}^{\times}$. Denote by
	\begin{displaymath}
	\Det_{H}^{l}:G\rightarrow (\mathbb{Z}_{l}^{\times})^{\tf}
	\end{displaymath}
	the resulting homomorphism to the torsion-free quotient. If $G$ is an $l$-cyclotomically full group of the sort considered in (ii), then we shall abbreviate $ w_{\cycl_{G}^{l}}(\Det_{H}^{l})$ as $w^{l}_{G/H}$.
	
	(iv) Let $M$ be a finitely generated $\Z_{l}$-module. Then denote by $\Rank_{l}(M)$ the $\Q_{l}$-dimension of $M\otimes_{\Z_{l}}\Q_{l}$.
\end{defi}

\begin{prop}\label{l15}
	Let $G$ a profinite group; $l$ be a prime number; $M$, $N$ finitely generated $\hat{\Z}$-modules with continuous $G$-actions such that $\Rank_{l}(M)=n$, $\Rank_{l}(N)=1$. Suppose that we are given a $\hat{\Z}$-bilinear $G$-equivariant map $f:M\times M\rightarrow N$ (i.e., $f(g(a),g(b))=g(f(a,b))$ for any $a,b\in M$, $g\in G$) that induces an isomorphism $M\otimes_{\Z_{l}}\Q_{l}\simeq \Hom(M\otimes_{\Z_{l}}\Q_{l}, N\otimes_{\Z_{l}}\Q_{l})$ of $\Q_{l}$-linear spaces. Moreover, suppose that the image of $\Det_{N}^{l}: G\to (\Z_{l}^{\times})^{\tf}$ is open. Then
	\begin{displaymath}
	w^{l}_{\Det_{N}^{l}}(\Det_{M}^{l})=n.
	\end{displaymath}
\end{prop}
\begin{proof}
	By assumption, $f$ induces a natural isomorphism
	\begin{displaymath}
	M\otimes_{\Z_{l}}\Q_{l}\simeq \Hom_{\Z_{l}}(M,\Z_{l})\otimes_{\Z_{l}}N\otimes_{\Z_{l}}\Q_{l}.
	\end{displaymath} 
	By taking the $n$-th exterior product, we obtain an isomorphism
	\begin{displaymath}
	\bigwedge^{n}M\otimes_{\Z_{l}}\Q_{l}\simeq \bigwedge^{n}\Hom_{\Z_{l}}(M,\Z_{l})\otimes_{\Z_{l}}N^{\otimes n}\otimes_{\Z_{l}}\Q_{l}.
	\end{displaymath}
	Note that since $\Q_{l}$ is of characteristic zero, we have a natural isomorphism
	\begin{displaymath}
	\bigwedge^{n}\Hom_{\Z_{l}}(M,\Z_{l})\otimes_{\Z_{l}}\Q_{l}\simeq \Hom_{\Z_{l}}(\bigwedge^{n}M,\Z_{l})\otimes_{\Z_{l}}\Q_{l}.
	\end{displaymath}
	Thus we obtain a $G$-equivariant isomorphism
	\begin{displaymath}
	\bigwedge^{n}M\otimes_{\Z_{l}}\bigwedge^{n}M\otimes_{\Z_{l}}\Q_{l}\simeq N^{\otimes n}\otimes_{\Z_{l}}\Q_{l},
	\end{displaymath}
	which implies that
	\begin{displaymath}
	2 w^{l}_{\Det_{N}^{l}}(\Det_{M}^{l})=n w^{l}_{\Det_{N}^{l}}(\Det_{N}^{l})=2n,
	\end{displaymath}
	as desired.
\end{proof}

\begin{prop}\label{l16}
	Let $G$ be a profinite group; $l$ a prime number; 
	\begin{displaymath}
	1\rightarrow M'\rightarrow M\rightarrow M''\rightarrow 1
	\end{displaymath}
	a $G$-equivariant exact sequence of finitely generated ${\Z}_{l}$-modules with $G$-actions. Let $\chi_{G}:G\rightarrow (\mathbb{Z}_{l}^{\times})^{\tf}$ be any character with open image. Then:
	
	(i) If any two of $ w^{l}_{\chi_{G}}(\Det_{M}^{l})$, $ w^{l}_{\chi_{G}}(\Det_{M'}^{l})$ and $ w^{l}_{\chi_{G}}(\Det_{M''}^{l})$ are well-defined, then all three are well-defined.
	
	(ii) In the situation of (i), 
	\begin{displaymath}
	w^{l}_{\chi_{G}}(\Det_{M}^{l})= w^{l}_{\chi_{G}}(\Det_{M'}^{l})+ w^{l}_{\chi_{G}}(\Det_{M''}^{l}).
	\end{displaymath}
\end{prop}
\begin{proof}
	Assertions (i) and (ii) follow immediately from the natural isomorphism
	\begin{displaymath}
	\bigwedge^{\Rank_{l}{M}}M\simeq \bigwedge^{\Rank_{l}{M'}}M'\otimes\bigwedge^{\Rank_{l}{M''}}M''
	\end{displaymath}
	of highest exterior powers.
\end{proof}

\begin{prop}\label{l17}
	Let $G$ be a profinite group; $l$ a prime number; $\chi_{G}:G\rightarrow (\mathbb{Z}_{l}^{\times})^{\tf}$ a character with open image; $H$ an \emph{open} subgroup of $G$; $\chi_{H}:H\rightarrow (\mathbb{Z}_{l}^{\times})^{\tf}$ the restriction of $\chi_{G}$ to $H$; $f_{G}:G\rightarrow (\mathbb{Z}_{l}^{\times})^{\tf}$ a character; $f_{H}:H\rightarrow (\mathbb{Z}_{l}^{\times})^{\tf}$ the restriction of $f_{G}$ to $H$. Then:
	
	(i) $ w^{l}_{\chi_{G}}(f_{G})$ is well-defined $\Longleftrightarrow$ $ w^{l}_{\chi_{H}}(f_{H})$ is well-defined.
	
	(ii) In the situation of (i),
	\begin{displaymath}
	w^{l}_{\chi_{G}}(f_{G})= w^{l}_{\chi_{H}}(f_{H}).
	\end{displaymath}
\end{prop}
\begin{proof}
	By definition, $\chi_{H}$, $\chi_{G}$, $f_{G}$, $f_{H}$ factor through $G^{\abt}$. Therefore, without loss of generality we may assume that $G$ is abelian and torsion-free. Suppose that $H$ is of index $n$. If $ w^{l}_{\chi_{G}}(f_{G})$ is well-defined, then it is immediate that $ w^{l}_{\chi_{H}}(f_{H})$ is well-defined. Now suppose that $ w^{l}_{\chi_{H}}(f_{H})$ is well-defined, and that  $(f_{H})^{a}=(\chi_{H})^{b}$. Let $g\in G$. Then $g^{n}\in H$. Thus, for any $g\in G$, $(f_{G})^{na}(g)=(\chi_{G})^{nb}(g)$, so $w^{l}_{\chi_{G}}(f_{G})= w^{l}_{\chi_{H}}(f_{H})$.
\end{proof}

By combining the preceding discussion about weights with the theory of Albanese varieties, one may compute the determinant morphism associated to the geometric fundamental group. In particular, one may compute the weight of the determinant morphism relative to the cyclotomic character.

\begin{theo}\label{l18}
	Let $k$ a field of characteristic zero; $X$ a smooth curve of type $(g,r)$ over $k$; $\bar{k}$ an algebraic closure of $k$; $l$ a prime number such that $X$ is $l$-cyclotomically full [cf. \cref{l43}(ii), where ``$X$'' is taken to be the $X$ of the present discussion]; $\cycl^{l}_{\Pi_{X}}:\Pi_{X}\to (\Z_{l}^{\times})^{\tf}$ the $l$-cyclotomic character; $\Delta_{X}\coloneqq\Pi_{X\times_{k}\bar{k}}$ the geometric fundamental group. Then
	
	(i) $\Delta_{X}$ is topologically finitely generated.
	
	(ii) $w^{l}_{\Pi_{X}/\Delta_{X}}$ is well-defined and equal to 
	\begin{displaymath}
		\begin{cases}
			2g+2r-2 & \tt{if } r\geq 1,\\
			2g & \tt{if } r=0.
		\end{cases}
	\end{displaymath}
	
	(iii) The morphism
	\begin{displaymath}
		\Det_{\Delta_{X}}^{l}:\Pi_{X}\rightarrow (\Z_{l}^{\times})^{\tf}
	\end{displaymath}
	factors through $\cycl_{\Pi_{X}}^{l}$.
	
	(iv) The following equality holds:
	\begin{displaymath}
		\Rank_{l}((\Delta_{X})^{\abt})=
		\begin{cases}
			2g+r-1 & \tt{if } r\geq 1,\\
			2g & \tt{if } r=0.
		\end{cases}
	\end{displaymath}
\end{theo}
\begin{proof}
	Since any smooth curve admits an open dense subscheme that is isomorphic to a hyperbolic smooth curve, without loss of generalization, one may assume that $X$ is hyperbolic. Then (i) follows immediately from \cite{MZK1}, Remark 1.2.2.
	
	
	For (ii), by \cref{l05}(ii) and \cref{l14}, it suffices to calculate $ w^{l}_{\Pi_{X}/T(A)}$ where $A$ is the Albanese variety of $X$. Let
	\begin{displaymath}
	1\rightarrow A_{\tor}\rightarrow A\rightarrow A_{\abe}\rightarrow 1
	\end{displaymath}
	be the unique exact sequence of $A$ such that $A_{\tor}$ is a torus and $A_{\abe}$ is an abelian variety. Since $A_{\tor}(\bar{k})$ is divisible, this exact sequence induces an exact sequence 
	\begin{displaymath}
	1\rightarrow T(A_{\tor})\rightarrow T(A)\rightarrow T(A_{\abe})\rightarrow 1
	\end{displaymath}
	of Tate modules.
	By \cref{l06} and the definition of Tate module, there exists a $l$-perfect pairing
	\begin{displaymath}
		T(A_{\abe})\times T(A_{\abe}) \rightarrow T(\G_{m}).
	\end{displaymath}
	By definition, the action of $\Pi_{X}$ on $T(\G_{m})\otimes_{\Z_{l}}\Q_{l}\cong \Q_{l}$ is given by \emph{cyclotomic} character, so $\Det^{l}_{T(\G_{m})}=\cycl^{l}_{\Pi_{X}}$. By \cref{l15}, we have $w^{l}_{\Pi_{X}/T(A_{\abe})} =\Rank_{l}(T(A_{\abe}))$. Since $\Rank_{l}(T(A_{\abe}))=2g$ [cf. \cref{l05}(i); \cite{MZK1}, Remark 1.2.2], we obtain
	\begin{displaymath}
		 w^{l}_{\Pi_{X}/T(A_{\abe})}=2g.
	\end{displaymath}
	
	By \cref{l17}, we may assume without loss of generality,
	by replacing $k$ by a suitable finite extension field of $k$, that
	$A_{\tor}$ is isomorphic to a product of copies of $\G_{m}$. Since the action of $\Pi_{X}$ on $T(\G_{m})\otimes_{\Z_{l}}\Q_{l}\cong \Q_{l}$ is given by \emph{cyclotomic} character, by \cref{l16} and \cref{l05}(ii), we conclude that
	\begin{displaymath}
		 w^{l}_{\Pi_{X}/T(A_{\tor})}=\text{max}\{2r-2,0\}.
	\end{displaymath}
	Thus, by \cref{l16},
	\begin{displaymath}
	 w^{l}_{\Pi_{X}/T(A)}=\text{max}\{2g+2r-2,2g\}.
	\end{displaymath}	
	Now (iii) follows immediately from (ii).
		
	Finally, we verify (iv). If $X$ is hyperbolic, then the statement follows immediately from \cite{MZK1}, Remark 1.2.2. If $(g,r)\in \{(0,0),(0,1)\}$, then (as is well-known) $\Delta_{X}$ is trivial. If $(g,r)=(0,2)$, then (as is well-known) $\Delta_{X}\cong \Zhat$. If $(g,r)=(1,0)$, then (as is well-known) $\Delta_{X}\cong \Zhat\times \Zhat$.
\end{proof}
	
\begin{coro}\label{l19}
	Let $\mathcal{C}:(X,Y,f,k,\bar{k},t,s)$ be a pointed virtual curve of type $(g,r)$ such that $k$ is of characteristic zero; $l$ a prime number. Then:
	
	(i) $\Delta_{\mathcal{C}}$ is topologically finitely generated.
	
	(ii) $ w^{l}_{\Pi_{\mathcal{C}}/\Delta_{\mathcal{C}}}$ exists and is equal to 
	\begin{displaymath}
		\begin{cases}
			2g+2r-2 & \tt{if } r\geq 1,\\
			2g & \tt{if } r=0.
		\end{cases}
	\end{displaymath}
	In particular, $ w^{l}_{\Pi_{\mathcal{C}}/\Delta_{\mathcal{C}}}$ is independent of the choice of section $s$.
	
	(iii) The morphism
	\begin{displaymath}
	\Det_{\Delta_{\mathcal{C}}}^{l}:\Pi_{\mathcal{C}}\rightarrow (\Z_{l}^{\times})^{\tf}
	\end{displaymath}
	factors through the $l$-cyclotomic (torsion-free) character $cycl_{\Pi_{X}}^{l}$.
	
	(iv) The following equality holds:
	\begin{displaymath}
		\Rank_{l}((\Delta_{\mathcal{C}})^{\abt})=
		\begin{cases}
			2g+r-1 & \tt{if } r\geq 1,\\
			2g & \tt{if } r=0.
		\end{cases}
	\end{displaymath}
\end{coro}
\begin{proof}
	\cref{l19} follows by combining \cref{l04}, \cref{l18}, and the fact that the cyclotomic character is functorial.
\end{proof}



\begin{lemm}\label{l44}
	Let $X$ be a smooth curve of type $(g,r)$ over an algebraically closed field $k$ of characteristic zero. Suppose that $r\geq 1$ and $(g,r)\notin {(0,1),(0,2)}$. Then there exists a finite set of \etale covers $X_{i}$, $i\in I$ of $X$ with type $(g_{i},r_{i})$ such that the greatest common divisor of $g_{i}+r_{i}-1$ is $1$.
\end{lemm}

\begin{proof}
	First, suppose that $g=0$. Then by assumption, $r\geq 3$. 
	Next, suppose that $g\geq 1$. Then by considering any \etale cover of $X$ that arises from an \etale cover of the unique smooth compactification of $X$ (i.e., any \etale cover of $X$ that is unramified over the cusps of $X$) of degree $2$ (which exists by \cite{SGA1}, Expos\'e XII, Corollaire 5.2), we may assume that $r\geq 2$. Therefore, for arbitrary $g$, we may assume without loss of generality that $r\geq 2$.
	
	By \cite{SGA1}, Expos\'e XII, Corollaire 5.2, for any positive integer $n$, there exists a cyclic \etale cover of $X$ that is totally ramified at two cusps and unramified at the other cusps. By the Hurwitz formula, one sees that the type of this \etale cover is $(ng,nr-2n+2)$. Now $\gcd(ng+nr-2n+1,g+r-1)=\gcd(1-n,g+r-1)$, so it suffices to choose an $n$ such that $n-1$ and $g+r-1$ is coprime.
\end{proof}

\begin{prop}\label{l22}
	Let $\mathcal{C}:(X,Y,f,k,\bar{k},t,s)$ be a pointed virtual curve of type $(g,r)$ such that $k$ is of characteristic zero; $l$ a prime number. Assume that $r\geq 1$. Then the $l$-cyclotomic character may be characterized as the unique group morphism $\chi:\Pi_{\mathcal{C}}\to (\Z_{l}^{\times})^{\tf}$ such that
	\begin{itemize}
		\item for any \etale covering $\mathcal{C}'$ of $\mathcal{C}$ (by abuse of notation, we do not distinguish $\chi$ from its restrictions to various $\Pi_{\mathcal{C}'}$), $ w^{l}_{\chi}(\Delta_{\mathcal{C}'})$ exists and is a positive \emph{even} number;
		\item the greatest common divisor of the elements of the set $\{w^{l}_{\chi}(\Delta_{\mathcal{C}'})\}_{\mathcal{C}'},$ where $\mathcal{C}'$ varies over the \etale coverings of $\mathcal{C}$, is $2$.
	\end{itemize}
	

\end{prop}
\begin{proof}
	By \cref{l19}(ii), $w^{l}_{\chi}(\Delta_{\mathcal{C}'})=2g(\mathcal{C}')+2r(\mathcal{C}')-2$, where $(g(\mathcal{C}'),r(\mathcal{C}'))$ is the type of $\mathcal{C}'$.
	Let $\bar{y}$ be a geometric point of $Y$. By \cref{l01}, there exists a natural isomorphism $\Pi_{X_{\bar{y}}}\simto \Delta_{\mathcal{C}}$. By \cref{l44}, there exists a set of Galois \etale covers $\{Z_{i}\}_{i\in I}$ of $X_{\bar{x}}$, where $Z_{i}$ is of type $(g_{i},r_{i})$, such that the greatest common divisor of the elements of the set $\{g_{i}+r_{i}-1\}_{i\in I}$ is $1$. For each $Z_{i}$, write $\Delta_{i}$ for the corresponding open subgroup of $\Delta_{\mathcal{C}}$. For each $\Delta_{i}$, choose an open subgroup $\Pi_{i}\subset \Pi_{\mathcal{C}}$ such that $\Pi_{i}\cap \Delta_{\mathcal{C}}=\Delta_{i}$.
	
	By \cref{l20}(ii) and \cref{l21}, for each $\Pi_{i}$, there exists a pointed virtual curve $\mathcal{C}_{i}:(X_{i},Y_{i},f_{i},k_{i},\bar{k},t_{i},s_{i})$ such that $\Pi_{\mathcal{C}_{i}}\simeq\Pi_{i}$. For each $Y_{i}$, choose an arbitrary geometric point $\bar{y}_{i}$ over $\bar{y}$. By definition, the type of $\mathcal{C}_{i}$ is equal to the type of $(X_{i})_{\bar{y}_{i}}$, which is isomorphic to $Z_{i}$ by construction.
	\cref{l22} now follows formally.

\end{proof}

\begin{prop}\label{l34}
	Let $\mathcal{C}:(X,Y,f,k,\bar{k},t,s)$ be a pointed virtual curve of type $(g,r)$ such that $k$ is of characteristic zero. Suppose that $k$ is $l$-cyclotomically full for some prime number $l$. Let $H$ be an open subgroup of $\Pi_{\mathcal{C}}$; $\mathcal{C''}$ an \etale covering of $\mathcal{C}$ such that the image of $\Pi_{\mathcal{C}''}$ in $\Pi_{\mathcal{C}}$ is $H$; $J\coloneqq H\cap \Delta_{\mathcal{C}}$. Then:
	
	(i) One may recover the type of $\mathcal{C''}$
	group-theoretically from the subgroup $H\subset \Pi_{\mathcal{C}}$ and the exact sequence of profinite groups
	\begin{displaymath}
	1\rightarrow \Delta_{\mathcal{C}} \rightarrow \Pi_{\mathcal{C}} \rightarrow \gal(\bar{k}/k) \rightarrow 1.
	\end{displaymath}
	In particular, it makes sense to write $(g(H),r(H))$ for the type of $\mathcal{C}''$.
	
	(ii) Suppose that $H'$ is an open subgroup of $\Pi_{\mathcal{C}}$ such that $H\cap \Delta_{\mathcal{C}}=H'\cap \Delta_{\mathcal{C}}$. Then $(g(H),r(H))=(g(H'),r(H'))$. Thus, it makes sense to write $(g(J),r(J))$ for $(g(H),r(H))$.
	
	(iii) Let $J_{X}$ be an open subgroup of $\Pi_{X}$ such that $J_{X}\cap \Delta_{\mathcal{C}}=J$. Write $J_{Y}$ for the image of $J_{X}$ in $\Pi_{Y}$; $f_{J}:X_{J}\to Y_{J}$ for the morphism between coverings of $X$ and $Y$ determined by $J_{X}\to J_{Y}$. Then $(g(J),r(J))$ is equal to the type of the family $f_{J}$ of curves. In particular, the type $(g(J),r(J))$ is independent of the choice of section $s$.
\end{prop}
\begin{proof}
	First, we consider (i). Observe that $k$ is $l$-cyclotomically full if and only if the image of $\Pi_{\mathcal{C}}$ by $\Det^{l}_{\Delta_{\mathcal{C}}}$ is nontrivial. Choose an arbitrary $l$ that satisfies this purely group-theoretic condition. Then observe that for any group homomorphism $\chi:\Pi_{\mathcal{C}}\to (\Z_{l}^{\times})^{\tf}$ and any \etale covering $\mathcal{C}'$ of $\mathcal{C}$, $w^{l}_{\chi}(\Delta_{\mathcal{C}'})$ may be recovered from the exact sequence in the statement of \cref{l34} and the image of $\Pi_{\mathcal{C}'}$ in $\Pi_{\mathcal{C}}$. Therefore, 
	by \cref{l22}, the $l$-cyclotomic character $\Pi_{\mathcal{C}}\to (\Z_{l}^{\times})^{\tf}$ may be reconstructed from this exact sequence.
	By \cref{l19}(ii)(iv), \cref{l20}(ii) and \cref{l21}, one may group-theoretically recover $2g(\mathcal{C''})+2r(\mathcal{C''})$ and $2g(\mathcal{C''})+r(\mathcal{C''})$ from the subgroup $H\subset \Pi_{\mathcal{C}}$ and this exact sequence. Therefore, by considering in addition whether or not $\Delta_{\mathcal{C}}$ is free, one may group-theoretically recover the type of $\mathcal{C''}$ from the subgroup $H\subset \Pi_{\mathcal{C}}$ and this exact sequence. 
	
	Next, we observe that (ii) follows immediately from \cref{l17}, by applying the proof of (i) in the case where we take ``$H$'' to be $H\cap H'$.
	
	Finally, we consider (iii). By (ii) and the fact that the type of a family of curves is preserved by passing to finite field extensions, we may assume that $\mathcal{C}''$ admits a morphism to $\mathcal{D}\coloneqq(X_{J},Y_{J},f_{J},k',\bar{k},t',s')$ --- where $t':Y_{J}\to \sp(k')$ is the geometrically connected morphism to the finite field extension of $k$ determined by $Y_{J}$, and $s'$ is a section of the naturally induced homomorphism $\Pi_{Y_{J}}\to G_{k'}$ --- such that the induced morphism of virtual fundamental groups is an isomorphism. By (i), the type of $\mathcal{C}''$ and $\mathcal{D}$ are equal. Since the type of $\mathcal{D}$ is the type of $f_{J}$, we complete the proof of (iii).
	
	
	
	
	
\end{proof}

\begin{defi}\label{l27}\quad
	
	(i) Let $\mathcal{C}:(X,Y,f,k,\bar{k},t,s)$ be a pointed virtual curve of type $(g,r)$ such that $k$ is of characteristic zero. Suppose that there exists a prime number $l$ such that $k$ is $l$-cyclotomically full. Then define the set of \emph{pre-cuspidal inertia subgroups} to be the set of closed subgroups $I\subset\Delta_{\mathcal{C}}$ that are isomorphic to $\hat{\Z}$ and satisfy the condition that for any open subgroup $J\subset \Delta_{\mathcal{C}}$ such that $J\neq IJ$, $r(J)<[IJ:J]r(IJ)$ [cf. \cref{l34}]. Define the set of \emph{quasi-cuspidal inertia subgroups} to be the set of pre-cuspidal inertia subgroups $I\subset\Delta_{\mathcal{C}}$ such that every open subgroup of $I$ is a pre-cuspidal inertia subgroup. Define the set of \emph{cuspidal inertia subgroups} to be the set of maximal elements in the set of quasi-cuspidal inertia subgroups with respect to the inclusion relation. Denote by $\Cusp(\mathcal{C})$ the set of \emph{conjugacy classes} of \emph{cuspidal inertia} subgroups in $\Delta_{\mathcal{C}}$.
	
	(ii) Let $k$ be a field of characteristic zero, $X$ a smooth hyperbolic curve over $k$. Then denote by $\cusp(X)$ the set of conjugacy classes of cuspidal inertia subgroups of $\Pi_{X}$, i.e., the intersections of decomposition groups of cusps with $\Delta_{X}$, the geometric fundamental group of $X$ [cf. \cref{Notations and Terminologies}, the discussion entitled ``Fundamental Groups'']. 
\end{defi}

\begin{theo}\label{l39}
	Let $\mathcal{C}:(X,Y,f,k,\bar{k},t,s)$ be a pointed virtual curve of type $(g,r)$ such that $k$ is of characteristic zero and $r\geq 1$. Suppose that there exists a prime number $l$ such that $k$ is $l$-cyclotomically full. Then the following properties hold:
	
	(i) The cardinality of $\Cusp(\mathcal{C})$ is $r$.
	
	(ii) $\Cusp(\mathcal{C})$ may be reconstructed group-theoretically from the exact sequence
	\begin{displaymath}
		1\rightarrow \Delta_{\mathcal{C}}\rightarrow \Pi_{\mathcal{C}}\rightarrow G_{k}\rightarrow 1,
	\end{displaymath}
	in a fashion that is functorial with respect to isomorphisms of exact sequences of profinite groups.
	Moreover, $\Cusp(\mathcal{C})$ is independent of the choice of section $s$ (where we observe that this statement of independence makes sense since $\Delta_{\mathcal{C}}$ is independent of the choice of section $s$) and is preserved by passing to finite field extensions of $k$.
	
	(iii) Suppose that $s$ arises from a $k$-rational point $y$ of $Y$. Then $\Cusp(\mathcal{C})=\Cusp(X_{y})$.
\end{theo}
\begin{proof}
	First, we observe that (i) follows immediately from (ii) and (iii). (Here, we note that one may replace $k$ by a finite extension of $k$ without affecting the validity of (i). Therefore, one may assume, without loss of generality, that there exists a $k$-rational point $y$ of $Y$.) Moreover, 
	(ii) follows immediately from \cref{l24} and \cref{l34}.
	
	Next, we consider (iii). 
	It is clear that any element of $\cusp(X_{y})$ is a quasi-cuspidal inertia subgroup. Therefore, it suffices to show that any quasi-cuspidal inertia subgroup of $\Delta_{\C}$ is a closed subgroup of some representative of an element of $\cusp(X_{y})$.
	
	By \cref{l02}, one may identify $\Pi_{\C}$ with $\Pi_{X_{y}}$. 
	For any open subgroup $\square$ of $\Delta_{\C}$, denote by $X_{\square}$ the corresponding curve over $\bar{k}$. Thus by \cref{l34}, $r(\square)=r(X_{\square})$.
	
	Let $I$ be a quasi-cuspidal inertia subgroup of $\Delta_{\C}$. Let $p$ be a prime. Denote by $I_{p}$ the unique $p$-Sylow subgroup of $I$.  Let $H$ be an open normal subgroup of $\Delta_{\C}$. We claim that the image of $I_{p}$ in $\Delta_{\C}/H$ is contained in the image of some cuspidal inertia group of $X_{y}$. Without loss of generality, one may assume that the image of $I_{p}$ in $\Delta_{\C}/H$ is nontrivial. By \cref{l27}(i), no nontrival intermediate subcovering (i.e., a covering $Z\to Z'$ that appears in some factorization $X_{H}\to Z\to Z'\to X_{HI_{p}}$) of the covering $X_{H}\to X_{HI_{p}}$ is unramified. 
	Thus by considering the subcovering corresponding to $H\subset H(I_{p})^{p}\subset HI_{p}=HI_{p}$, we conclude (since $I_{p}$ is cyclic pro-$p$) that some cuspidal inertia subgroup of $X_{HI_{p}}$ surjects onto $HI_{p}/H$. Since $H$ is arbitrary, by \cite{nodnon}, Lemma 1.5, we thus conclude that there exists a unique cuspidal inertia subgroup $J_{p}$ of $X_{y}$ such that $I_{p}\subset J_{p}$. 
	
	Note that for any prime  $q\neq p$, $I_{q}$ commutes with $I_{p}$, hence (by \cite{nodnon}, Lemma 1.5) is contained in $J_{p}$. Again, by \cite{nodnon}, Lemma 1.5, this implies that $J_{q}=J_{p}$. Since $p$, $q$ are arbitrary, we thus conclude that there exists a unique cuspidal inertia group $J$ of $X_{y}$ such that $I\subset J$. This completes the proof of (iii).
	
	
	
\end{proof}



\section{Characterization of Geometric Fundamental Groups}\label{Characterization of Geometric Fundamental Groups}

In this section, we discuss how to reconstruct the geometric fundamental subgroup of the virtual fundamental group of a pointed virtual curve over either a number field or a mixed characteristic local field. First, we consider the number field case.

\begin{theo}\label{l38}
	Let $k$ be a number field; $\mathcal{C}$ a pointed virtual curve over $k$. Then the geometric fundamental group $\Delta_{\mathcal{C}}$ may be characterized as the unique maximal normal closed topologically finitely generated subgroup of $\Pi_{\mathcal{C}}$.
\end{theo}
\begin{proof}
	By \cref{l19}(i), $\Delta_{\mathcal{C}}$ is a normal closed topologically finitely generated subgroup. Thus it suffices to show that for any normal closed topologically finitely generated subgroup $H$ of $\Pi_{\mathcal{C}}$, the image of $H$ in $G_{k}$ is trivial. But this follows from the fact that $G_{k}$ is \emph{very elastic}, i.e., elastic and \emph{not} topologically finitely generated [cf. \cite{MZK3}, Theorem 1.7(iii)].
\end{proof}


In the case of a mixed characteristic local field $k$, we will use \cref{l25} to calculate the absolute degree of $k$ from the abstract topological group $G_{k}$.

\begin{prop}\label{l25}
	Let $k$ be a mixed characteristic local field of residue characteristic $p$; $l\neq p$ a prime number. Denote by $[k:\Q_{p}]$ the dimension of $k$ over $\Q_{p}$. Then
	\begin{displaymath}
		\Rank_{p}((G_{k})^{\abt})-\Rank_{l}((G_{k})^{\abt})=[k:\Q_{p}].
	\end{displaymath}
	Moreover, $G_{k}$ is $l$-cyclotomically and $p$-cyclotomically full.
\end{prop}
\begin{proof}
	This is an immediate consequence of local class field theory [cf. the discussion preceeding \cite{NSW}, 7.2.11; \cite{Neu}, II.5.7(i)].
\end{proof}

\begin{prop}\label{l29}
	Let $l$ be a prime number, $k$ an $l$-cyclotomically full field of characteristic zero; $\mathcal{C}:(X,Y,f,k,\bar{k},t,s)$ a pointed virtual curve such that $Y$ is smooth over $k$. Suppose that there exist morphisms $u:Z\to Y$, $v:X\to Z$ where $u$ is proper, \emph{smooth}, and geometrically connected of relative dimension one, and $v$ is an open immersion with dense image, such that $f=u\circ v$ [cf. \cref{Notations and Terminologies}, the discussion entitled ``Schemes and Curves''], and the reduced closed subscheme $D$ determined by the complement of the image of $X$ in $Z$ is (nessesarily finite) \etale over $Y$. Let $D'$ be a reduced closed subscheme of $Z$ that contains $D$ such that $D'$ is (nessesarily finite) \etale over $Y$. Denote by $X'$ the complement of $D'$ in $Z$. Write $i:X'\to X$ for the natural open immersion; $f'\coloneqq f\circ i$; $r$ for the degree of $D$ over $Y$; $r'$ for the degree of $D'$ over $Y$. Then the following properties hold:
	
	(i) $\mathcal{C}'\coloneqq(X',Y,f',k,\bar{k},t,s)$ is a pointed virtual curve, and $\mathcal{F}\coloneqq(i,\id,\id):\mathcal{C}'\to \mathcal{C}$ is a morphism of virtual varieties.
	
	(ii) The induced maps $\pi_{\F}:\Pi_{\C'}\rightarrow \Pi_{\C}$, $\delta_{\F}:\Delta_{\C'}\rightarrow \Delta_{\C}$ are surjections. Moreover, the kernel of $\delta_{\F}$ and $\pi_{\F}$ are equal and generated by the subgroups in a collection of $r'-r$ conjugacy classes of cuspidal inertia subgroups [cf. \cref{l27}].
	
	(iii) Assume further that $k$ is a mixed characteristic local field of residue characteristic $p$. Then $\pi_{\F}$ induces an isomorphism
	\begin{displaymath}
	(\Pi_{\C'})^{\abt}\simeq(\Pi_{\C})^{\abt}.
	\end{displaymath}	
\end{prop}
\begin{proof}
	
	First, we observe that (i) follows immediately from the definitions.
	
	By \cref{l24} and \cref{l39}(ii), the validity of (ii) is independent of the choice of section $s$. By \cref{l39}(ii), one may replace $k$ by a finite field extension of $k$. Therefore, without loss of generality, one may assume that $s$ arises from a $k$-rational point $y$ of $Y$. By \cref{l02} and \cref{l39}(iii), one may replace various virtual fundamental groups (and cuspidal inertia groups of virtual fundamental groups) by the \etale fundamental groups of the respective fibers over $y$ (and the cuspidal inertia groups of the respective fibers over $y$). Then (ii) follows from well-known general facts concerning \etale fundamental groups of curves.
	
	Finally, (iii) follows from (ii), \cref{l04}, \cref{l05}(i)(ii) (where we note the natural functorial relationship between \cref{l05}(i) and \cref{l05}(ii)), and the fact that $k$ is $l'$-cyclotomically full for arbitrary prime $l'$ [cf. \cref{l25}]. 
\end{proof}


\begin{defi}\label{l66}\quad
	
	(i) We refer to any morphism of pointed virtual curves that is isomorphic to some $\F$ of the sort constructed in \cref{l29} as a \emph{decuspidalization} of pointed virtual curves.
	
	(ii) Suppose that in the notation of \cref{l29}, $D'\setminus D$ is of degree $n$ over $Y$ (i.e., $n\coloneqq r'-r$, where $r,r'$ are defined in \cref{l29}). Then we say that $\F$ is a \emph{degree $n$ decuspidalization}. If $n=1$, then we refer to the unique conjugacy class of cuspidal inertia groups that appears in \cref{l29}(ii) as the conjugacy class (of cuspidal inertia groups) determined by $\F$.
\end{defi}

\begin{remm}
	In [\cite{MZK3}, Definition 4.2], the word ``decuspidalization'' (written as ``de-cuspidalization'') is used to describe the procedure of removing a \emph{single} cusp. However, in this paper, we allow the number of removed cusps to be arbitrary.
\end{remm}

\begin{prop}\label{l65}
	We maintain the notation of \cref{l29}. Suppose that $\F$ is a degree $1$ decuspidalization. Let $I\subset \Pi_{\mathcal{C}}$ be one of the representatives of the conjugacy class determined by $\F$. Assume further that $s$ arises from a $k$-rational point $y\in Y$. Then:
	
	(i) The morphism
	\begin{displaymath}
		\pi_{\F}:\Pi_{\C'}\rightarrow \Pi_{\C}
	\end{displaymath}
	is naturally isomorphic to
	\begin{displaymath}
		\Pi_{(X')_{y}}\to \Pi_{X_{y}}.
	\end{displaymath}
	Moreover, $i_{y}:(X')_{y}\to X_{y}$ is a decuspidalization in the sense of \cite{MZK3}, Definition 4.2(i).
	
	(ii) $I$, considered as a closed subgroup of $\Pi_{(X')_{y}}$, is a representative of the conjugacy class of inertia subgroups that determined by the decuspidalization $i_{y}$.
\end{prop}
\begin{proof}
	Indeed, (i) follows from \cref{l02}, while (ii) follows from (i) and the proof of \cref{l29}(ii).
\end{proof}

\begin{prop}\label{l64}
	Let $\F= (i,\id,\id):\C'=(X',Y,f',k,\bar{k},t,s)\to\C=(X,Y,f,k,\bar{k},t,s)$ be a decuspidalization as in \cref{l29}; $\D:(X'',Y'',f'',k'',\bar{k}'',t'',s'')$ a pointed virtual curve; $\mathcal{G}\coloneqq(h_{X},h_{Y},h_{k}):\D\to \C$ an \etale covering. Then there exists a commutative diagram
	\begin{displaymath}
		\xymatrix{
			{\D'}\ar^{{\F'}}[r]\ar^{{\mathcal{G}'}}[d] & {\D}\ar^{{\mathcal{G}}}[d]\\
			{\C'}\ar^{{\F}}[r] &{\C}
		}
	\end{displaymath}
	of morphisms of pointed virtual curves such that $\F'$ is a decuspidalization, $\mathcal{G}'$ is an \etale covering, and the induced diagram
	\begin{displaymath}
		\xymatrix{
			\Pi_{\D'}\ar^{\pi_{\F'}}[r]\ar^{\pi_{\mathcal{G}'}}[d] & \Pi_{\D}\ar^{\pi_{\mathcal{G}}}[d]\\
			\Pi_{\C'}\ar^{\pi_{\F}}[r] &\Pi_{\C}
		}
	\end{displaymath}
	of morphisms of virtual fundamental groups is cartesian.
\end{prop}
\begin{proof}
	Let $\D'$ be the virtual variety
	\begin{displaymath}
		(X'\times_{X}X'',Y'',f''\circ (i\times_{X}X''),k'',\bar{k''},t'',s'');
	\end{displaymath}
	$\F':\D'\to \D$ (respectively, $\mathcal{G}':\D'\to\C'$) the morphism $(i\times_{X}X'',\id,\id)$ (respectively, $(h_{X}\times_{X}X',h_{Y},h_{k}))$. One verifies immediately that $\mathcal{G}'$ is an \etale covering. 
	
	By definition, there exists a proper family $t:Z\to Y$ of curves and two reduced closed subschemes $D$, $D'$ of $Z$ such that $D\subset D'$, $X$ (respectively, $X'$) is the complement of $D$ (respectively, $D'$) in $Z$, and both $D$ and $D'$ are \etale over $Y$. Note that this data determines a morphism $i^{\log}:(X')^{\log}\to X^{\log}$ of \emph{regular} log schemes such that $X$ (respectively, $X'$) is the interior of $X^{\log}$ (respectively, $(X')^{\log}$), and $Z$ is the underlying scheme of $X^{\log}$ (respectively, $(X')^{\log}$). By log purity [cf., e.g.,  \cite{MZK4}, Theorem B] and the fact that $h_{X}$ is an \etale covering, there exists a regular log scheme $W^{\log}$ and a commutative diagram
	\begin{displaymath}
		\xymatrix{
			W\ar^{t''}[r]\ar^{s}[d]& Y''\ar^{h_{Y}}[d]\\
			Z\ar^{t}[r]& Y
		}
	\end{displaymath}
	--- where $W$ is the underlying scheme of $W^{\log}$ --- such that
	\begin{itemize}
		\item the log structure on $W^{\log}$ is determined by $s^{-1}(D)$;
		\item the interior of $W^{\log}$ is $X''$;
		\item $t''$ is smooth and proper;
		\item $f''$ is induced by $t''$;
		\item $s^{-1}(D)$ and $s^{-1}(D')$ are \etale over $Y''$.
	\end{itemize}
	 One thus verifies immediately that $\F'$ is a decuspidalization. Since $h_{X}$ is an \etale covering, $h_{X}\times_{X}X'$ is an \etale covering, and the commutative diagram
	\begin{displaymath}
	\xymatrix{
		\Pi_{X'\times_{X}X''}\ar^{(i\times_{X}X'')_{*}}[rr]\ar^{(h_{X}\times_{X}X')_{*}}[d] & &\Pi_{X''}\ar^{(h_{X})^{*}}[d]\\
		\Pi_{X'}\ar^{i_{*}}[rr]& &\Pi_{X}
	}
	\end{displaymath}
	is cartesian. Thus, one concludes that the diagram
	\begin{displaymath}
	\xymatrix{
		\Pi_{\D'}\ar^{\pi_{\F'}}[r]\ar^{\pi_{\mathcal{G}'}}[d] & \Pi_{\D}\ar^{\pi_{\mathcal{G}}}[d]\\
		\Pi_{\C'}\ar^{\pi_{\F}}[r] &\Pi_{\C}
	}
	\end{displaymath}
	is cartesian, as desired.
\end{proof}

\begin{prop}\label{l30}
	Let $k$ be a mixed characteristic local field of residue characteristic $p$; $X$ a smooth hyperbolic curve over $k$; $s:G_{k}\rightarrow \Pi_{X}$ a section. Then, in the notation of \cref{Notations and Terminologies}, the discussion entitled ``Configuration Spaces'', and \cref{l41}(iii):
	
	(i) $[\pr^{2/1}_{X},s]$ and $[\pr^{1}_{X},s]$ are pointed virtual curves.
	
	(ii) The natural morphism $\F_{X}:[\pr^{2/1}_{X},s]\rightarrow [\pr^{1}_{X},s]$ of pointed virtual curves is a decuspidalization of pointed virtual curves.
	
	(iii) $\F_{X}$ induces a natural isomorphism
	\begin{displaymath}
		(\Pi_{[\pr^{2/1}_{X},s]})^{\abt}\simeq(\Pi_{[\pr^{1}_{X},s]})^{\abt}.
	\end{displaymath}
	
	(iv) Let $H$ be an open subgroup of $\Pi_{[\pr^{2/1}_{X},s]}$ that contains the kernel of the naturally induced morphism $\Pi_{[\pr^{2/1}_{X},s]}\to (\Pi_{[\pr^{2/1}_{X},s]})^{\abt}$. Denote by $J$ the image of $H$ in $\Pi_{[\pr^{1}_{X},s]}$. Then the morphism
	\begin{displaymath}
		(\F_{X}|_{H})^{\abt}: H^{\abt}\to J^{\abt}
	\end{displaymath}
	induced by the natural morphism
	\begin{displaymath}
		\F_{X}|_{H}: H\to J
	\end{displaymath}
	is an isomorphism.
\end{prop}
\begin{proof}
	Assertion (i), (ii) follow immediately [cf. \cite{MZK1}, Remark 2.1.1]. Assertion (iii) follows from assertion (ii) and \cref{l29}(iii). 
	The condition of assertion (iv) implies that $H$ contains the kernel of $\pi_{\F_{X}}: \Pi_{[\pr^{2/1}_{X},s]}\to\Pi_{[\pr^{1}_{X},s]}$. Therefore, by \cref{l64}, there exists a decuspidalization $\mathcal{F}:\mathcal{D}'\to\mathcal{D}$ such that $H\simeq \Pi_{\mathcal{D}'}$, $J\simeq \Pi_{\mathcal{D}}$, and $\F_{X}|_{H}$ may be identified with $\pi_{\mathcal{F}}$, relative to the above isomorphisms. Thus, by \cref{l29}(iii), one concludes that $(\F_{X}|_{H})^{\abt}$ is an isomoprhism.
\end{proof}

\begin{defi}\label{l31}
	Let $k$ be a mixed characteristic local field of residue characteristic $p$; $X$ a smooth hyperbolic curve over $k$; $s:G_{k}\rightarrow \Pi_{X}$ a section. Then we shall say that an open normal subgroup $H$ of $\Pi_{[\pr^{2/1}_{X},s]}$ is \emph{of admissible type} if $H$ contains the kernel of the naturally induced morphism $\Pi_{[\pr^{2/1}_{X},s]}\to (\Pi_{[\pr^{2/1}_{X},s]})^{\abt}$.
\end{defi}

\begin{prop}\label{l32}
	Let $k$ be a mixed characteristic local field of residue characteristic $p$; $X$ a smooth hyperbolic curve over $k$; $s:G_{k}\rightarrow \Pi_{X}$ a section; $l\neq p$ a prime number. Then for any open normal subgroup $H\subset \Pi_{[\pr^{2/1}_{X},s]}$ of admissible type corresponding to an \etale covering $\D$ of $[\pr^{2/1}_{X},s]$, 
	\begin{displaymath}
		\Rank_{p}((\Pi_{\D})^{\abt})-\Rank_{l}((\Pi_{\D})^{\abt})=[k_{H}:\Q_{p}],
	\end{displaymath}
	where $k_{H}$ is the finite field extension of $k$ that corresponds to the image of $H$ in $G_{k}$.
\end{prop}
\begin{proof}
	By \cref{l30}(iii) and \cref{l31}, $H$ is the inverse image of a normal open subgroup $J\subset \Pi_{[\pr^{1}_{X},s]}$ by the morphism $\Pi_{[\pr^{2/1}_{X},s]}\rightarrow \Pi_{[\pr^{1}_{X},s]}$. By \cref{l33}(i) below, $J$ corresponds to an \etale covering $Y\rightarrow X$ such that $J\simeq \Pi_{[\pr^{1}_{J},s_{H}]}$, where
	\begin{displaymath}
		\pr^{1}_{J}:X\times_{k}Y\simeq(X\times_{k}k_{H})\times_{k_{H}}Y\rightarrow X\times_{k}k_{H}
	\end{displaymath}
	denotes the first projection, and $s_{H}:G_{k_{H}}\rightarrow \Pi_{X\times_{k}k_{H}}$ the restriction of $s$ to $G_{k_{H}}$. By \cref{l29}, $H\simeq \Pi_{[f,s_{H}]}$, where $f$ denotes the morphism naturally induced by the first projection
	\begin{displaymath}
		(X\times_{k}Y)\times_{(X\times_{k}X)}X_{2}\rightarrow X\times_{k}k_{H}.
	\end{displaymath}
	Denote $[\pr^{1}_{J},s_{H}]$ by $\D_{J}$. Then by \cref{l30}(iv),
	\begin{align*}
	\Rank_{p}((\Pi_{\D})^{\abt})-\Rank_{l}((\Pi_{\D})^{\abt}) &=\\
	\Rank_{p}(H^{\abt})-\Rank_{l}(H^{\abt}) &=\\
	\Rank_{p}((J)^{\abt})-\Rank_{l}((J)^{\abt}) &=\\
	\Rank_{p}((\Pi_{\D_{J}})^{\abt})-\Rank_{l}((\Pi_{\D_{J}})^{\abt}) &.
	\end{align*}
	Moreover, by \cref{l33}(iii) below, 
	\begin{displaymath}
		\Rank_{p}(\Pi_{\D_{J}}^{\abt})-\Rank_{l}(\Pi_{\D_{J}}^{\abt})=[k_{H}:\Q_{p}].
	\end{displaymath}
	This completes the proof of \cref{l32}.
\end{proof}

\begin{lemm}\label{l33}
	Let $k$ be a field of characteristic zero; $\bar{k}$ an algebraic closure of $k$; $X$, $Y$ smooth hyperbolic curves over $k$; $[f,s]$ a pointed virtual curve with base $k$, such that $f$ is the projection $X\times_{k}Y\rightarrow X$, and $Y$ is an \etale covering of $X$. Then the following hold:
	
	(i) The exact sequence
	\begin{displaymath}
		1\rightarrow \Delta_{[f,s]}\rightarrow \Pi_{[f,s]}\rightarrow G_{k}\rightarrow 1
	\end{displaymath}
	is naturally isomorphic to 
	\begin{displaymath}
		1\rightarrow \Pi_{Y\times_{k}\bar{k}}\rightarrow \Pi_{Y}\rightarrow G_{k}\rightarrow 1.
	\end{displaymath}
	This isomorphism is induced by the projection
	\begin{displaymath}
		X\times_{k}Y\rightarrow Y.
	\end{displaymath}

	(ii) Assume that there exists a prime number $l$ such that $k$ is $l$-cyclotomically full. Then $\cusp([f,s])$ is naturally isomorphic to the set of conjugacy classes of inertia subgroups of the cusps of $Y$.
	
	(iii) Suppose further that $k$ is a mixed characteristic local field of residue characteristic $p$. Let $l\neq p$ be a prime number. Then
	\begin{displaymath}
	\Rank_{p}((\Pi_{Y})^{\abt})-\Rank_{l}((\Pi_{Y})^{\abt})=[k:\Q_{p}].
	\end{displaymath}
\end{lemm}
\begin{proof}
	Assertion (i) follows from the fact that
	\begin{displaymath}
		\Pi_{(X\times_{k}Y)\times_{k}\bar{k}}\simeq \Pi_{X\times_{k}\bar{k}}\times \Pi_{Y\times_{k}\bar{k}},
	\end{displaymath}
	which follows from the fact that both $X$ and $Y$ are the interiors of log smooth curves over $k$, and [\cite{Hos}, Theorem 2].
	To show assertion (ii), one may assume without loss of generality that $Y$ admits a $k$-rational point. Then assertion (ii) follows from \cref{l39}(ii)(iii).
	Assertion (iii) follows from [\cite{MZK6}, Lemma 1.1.4(ii)]
\end{proof}

\begin{defi}\label{l46}
	Let $k$ be a mixed characteristic local field of residue characteristic $p$; $X$ a smooth hyperbolic curve over $k$; $s:G_{k}\rightarrow \Pi_{X}$ a section. Then we say that an open normal subgroup $H\subset \Pi_{[\pr^{2/1}_{X},s]}$ that is of admissible type is \emph{of field extension type} if
	\begin{displaymath} 
		\Rank_{p}(H^{\abt})-\Rank_{l}(H^{\abt})=
		[\Pi_{[\pr^{2/1}_{X},s]}:H](\Rank_{p}((\Pi_{[\pr^{2/1}_{X},s]})^{\abt})-\Rank_{l}((\Pi_{[\pr^{2/1}_{X},s]})^{\abt})).
	\end{displaymath}	
	Let $J$ be a topologically finitely generated closed normal subgroup of $\Pi_{X}$. Then we say that $J$ is \emph{of geometric type} if for any subgroup $
	H\subset \Pi_{[\pr^{2/1}_{X},s]}$ of field extension type, $J\subset H$.
\end{defi}



\begin{theo}\label{l37}
	Let $k$ be a mixed characteristic local field of residue characteristic $p$; $X$ a smooth hyperbolic curve over $k$; $s:G_{k}\rightarrow \Pi_{X}$ a section. Then the subgroup $\Delta_{[\pr^{2/1}_{X},s]}$ may be purely group-theoretically reconstructed from the abstract group $\Pi_{[\pr^{2/1}_{X},s]}$, in a fashion that is functorial with respect to isomorphisms of profinite groups. 
\end{theo}

\begin{proof}
	
	By \cref{l32} and \cref{l46}, any subgroup of $\Pi_{[\pr^{2/1}_{X},s]}$ that is of field extension type contains $\Delta_{[\pr^{2/1}_{X},s]}$, as well as the commutator of $\Pi_{[\pr^{2/1}_{X},s]}$. Therefore, $\Delta_{[\pr^{2/1}_{X},s]}$ is of geometric type. 
	We claim that any subgroup $J$ of  $\Pi_{[\pr^{2/1}_{X},s]}$ that is of geometric type is contained in $\Delta_{[\pr^{2/1}_{X},s]}$. Indeed, consider the image $I$ of $J$ in $G_{k}\simeq \Pi_{[\pr^{2/1}_{X},s]}/\Delta_{[\pr^{2/1}_{X},s]}$. Since $I$ is a topologically finitely generated closed normal subgroup of $G_{k}$, by the elasticity of $G_{k}$ [cf. \cite{MZK3}, Theorem 1.7(ii)], $I$ is either open or trivial. On the other hand, $I$ is contained in the intersection of the images in $G_{k}$ of the open subgroups $H\subset\Pi_{[\pr^{2/1}_{X},s]}$ of field extension type, hence is of infinite index in $G_{k}$ (since $G_{k}^{\abt}$ is infinite). Thus, $J\subset \Delta_{[\pr^{2/1}_{X},s]}$, as desired. Therefore, $\Delta_{[\pr^{2/1}_{X},s]}$ may be characterized as the maximal subgroup of $\Pi_{[\pr^{2/1}_{X},s]}$ that is of geometric type.
\end{proof}

\section{Automorphisms of 2-Configuration Spaces}\label{Automorphisms of 2-Configuration Spaces}

In \cref{Weights of Determinant Maps} and \cref{Characterization of Geometric Fundamental Groups}, we discussed various ways to characterize various objects of geometric origin. In this section, we first construct the function field of a smooth, geometrically connected curve of type $(0,3)$ over either a mixed-characteristic local field or a number field from the virtual fundamental group associated to a Galois section of the curve.
Then, by generalizing the previous construction, we construct various scheme-theoretic objects, such as $\mathbf{F}$-graphs (where $\mathbf{F}$ is the category of fields, [cf. \cref{Notations and Terminologies}, the discussion entitled ``Graphs and Categories'']) whose vertices map to function fields of curves of type $(0,n)$, where $n\geq 3$, over either a mixed-characteristic local field or a number field, from certain virtual fundamental groups.



\begin{prop}\label{l48}
	Let $k$ be a field of characteristic zero such that there exists a prime number $l$ such that $k$ is $l$-cyclotomically full; $\bar{k}$ an algebraic closure of $k$; $X$ a smooth hyperbolic curve over $k$; $s:G_{k}\rightarrow \Pi_{X}$ a section; $\diag_{X}\in \Cusp([\pr^{2/1}_{X},s])$ the conjugacy class of cuspidal inertia subgroups determined by the decuspidalization $[\pr^{2/1}_{X},s]\to [\pr^{1}_{X},s]$ of pointed virtual curves [cf. \cref{l66}(i), \cref{l30}]; $H$ the normal closed subgroup of $\Pi_{[\pr^{2/1}_{X},s]}$ generated by $\diag_{X}$. Then the following exact sequences
	\begin{displaymath}
	1\rightarrow \Delta_{[\pr^{2/1}_{X},s]}/H\rightarrow \Pi_{[\pr^{2/1}_{X},s]}/H\rightarrow G_{k}\rightarrow 1
	\end{displaymath}
	\begin{displaymath}
	1\rightarrow \Delta_{[\pr^{1}_{X},s]}\rightarrow \Pi_{[\pr^{1}_{X},s]}\rightarrow G_{k}\rightarrow 1
	\end{displaymath}
	\begin{displaymath}
	1\rightarrow \Pi_{X\times_{k}\bar{k}}\rightarrow \Pi_{X}\rightarrow G_{k}\rightarrow 1
	\end{displaymath}
	are naturally isomorphic. Moreover, there exists a natural bijection between $\cusp([\pr^{1}_{X},s])\simeq\cusp([\pr^{2/1}_{X},s])\setminus\{\diag_{X}\}$ and the set of conjugacy classes of inertia subgroups of cusps of $X$.

\end{prop}
\begin{proof}
	It follows from \Cref{l29}(ii) that the first and the second exact sequences are isomorphic. It follows from \Cref{l33} that the second and the third exact sequences are isomorphic, and that there exists a natural bijection between $\cusp([\pr^{1}_{X},s])$ and the set of conjugacy classes of inertia subgroups of cusps of $X$.
\end{proof}

	
	In the following proposition, we recall some well-known results from anabelian geometry.
	
\begin{prop}\label{l72}
	Let $k$ be either a number field or a mixed-characteristic local field; $X$ a smooth curve of strictly Belyi type  over $k$. Then there exists a group-theoretic algorithm for reconstructing the function field $K(X)$ of $X$, as well as the set of valuations of $K(X)$ that arise from the cusps of $X$ and the inclusion $k\injto K(X)$ of fields, from the exact sequence of profinite groups
	\begin{displaymath}
		1\to \Delta_{X}\to \Pi_{X}\to G_{k}\to 1,
	\end{displaymath}
	in a fashion that is functorial with respect to isomorphisms of exact sequences of profinite groups.
	
	
\end{prop}
\begin{proof}
	This follows immediately from \cite{MZK2}, Theorem 1.9, and \cite{MZK2}, Corollary 1.10. 
	
	
\end{proof}


\begin{lemm}\label{l55}
	Let $k$ be either a number field or a mixed-characteristic local field; $X$ a smooth curve of type $(0,n)$ over $k$ that is defined over a number field. Suppose that $n\geq 3$. Then $X$ is of strictly Belyi type [cf. \cite{MZK5}, Definition 3.5].
\end{lemm}
\begin{proof}
	This follows immediately from \cite{MZK5}, Definition 3.5.
\end{proof}

Generally speaking, in the situation of \cref{l48}, one cannot group-theoretically characterize the element $\diag_{X}\in\C$. However, in the special case of type $(0,3)$, one may show that \emph{every} class of cuspidal inertia subgroups that is fixed by the conjugation action of $G_{k}$ may be regarded as the class that represents the diagonal. 


First, we review basic facts concerning the automorphism groups of curves and $2$-configuration spaces.

\begin{lemm}\label{l47}
	Let $k$ be a field of characteristic zero; $X$ a smooth curve over $k$ that is isomorphic to $ \mathbf{P}^{1}_{k}\setminus\{0,1,\infty\}$. 
	
	(i) One may regard $X$ as the moduli space $\mathcal{M}_{0,4}$ of four-pointed genus zero curves with ordered marked points. In particular, there exists a natural $\mathbb{S}_{4}$-action on $X$. Denote the image of $\mathbb{S}_{4}$ in the group $\Aut_{k}(X)$ of $k$-automorphisms of $X$ by $H$.
	
	(ii) One may regard $X_{2}$ as the moduli space $\mathcal{M}_{0,5}$ of five-pointed genus zero curves with ordered marked points. In particular, there exists a natural faithful $\mathbb{S}_{5}$-action on $X_{2}$. Denote the image of $\mathbb{S}_{5}$ in the group $\Aut_{k}(X_{2})$ of $k$-automorphisms of $X_{2}$ by $G$.
	
	(iii) One may regard the projection $\pr^{2/1}_{X}$ as the natural morphism $\mathcal{M}_{0,5}\to \mathcal{M}_{0,4}$ given by forgetting the fifth marked point.

	
	(iv)  $H=\Aut_{k}(X)$; moreover, $H$ is isomorphic to $\mathbb{S}_{3}$ and acts faithfully on the set of $3$ cusps in $X$.
	
	(v) The subgroup $I$ of $G$ that consists of $f\in G$ such that there exists $h\in H$ and a commutative diagram
	\begin{displaymath}
		\xymatrix{
			X_{2}\ar^{f}[r]\ar^{\pr^{2/1}_{X}}[d] &X_{2}\ar^{\pr^{2/1}_{X}}[d]\\
			X\ar^{h}[r]& X
		}
	\end{displaymath}
	is isomorphic to $\mathbb{S}_{4}$.
	
	(vi) Note that for any $f\in I$, the $h\in H$ that satisfies the commutativity property in (v) is unique. In particular, we obtain a homomorphism $u:I\to H$, whose kernel $J$ is isomorphic to $\Z/2\Z\times \Z/2\Z$. If $s:G_{k}\to \Pi_{X}$ is a section, then elements of $J$ may be regarded as automorphisms of $[\pr^{2/1}_{X},s]$.
	
	(vii) Let $x\in X$ be a $k$-rational point. Then $J$ permutes the four cusps in the fiber $(X_{2})_{x}\simeq X\setminus x$. The image of the induced outer homomorphism $J\rightarrow \mathbb{S}_{4}$ is the intersection of all Sylow-$2$ subgroups in $\mathbb{S}_{4}$. 
\end{lemm}
\begin{proof}
	The assertions of \cref{l47} follow immediately from various definitions involved.
\end{proof}


\begin{defi}
	Let $k$ be a field of characteristic zero; $X$ a smooth curve over $k$ that is isomorphic to $ \mathbf{P}^{1}_{k}\setminus\{0,1,\infty\}$. Denote by $\sxxx$ (respectively, $\sx$; $\sxx$; $\vx$) the group $G$ (respectively, $H$, $I$, $J$) constructed in \cref{l47}. Denote by $F_{X}:\sxx\to \sx$ the morphism $u$ constructed in \cref{l47}.
\end{defi}

\begin{lemm}\label{l67}
	Let $k$ be a field; $k'$ a finite \emph{Galois} field extension of $k$. Let $X$, $Y$ be schemes over $k$ such that there exists a $k'$-isomorphism $f:X\times_{k}k'\simto Y\times_{k}k'$. Then there exists a(n) (necessarily unique) $k$-isomorphism $h:X\simto Y$ such that $f=h\times_{k}k'$ if and only if for any element $\alpha\in\gal(k'/k)$, $(Y \times_{k}\alpha)\circ f=f\circ (X\times_{k}\alpha)$.
	
\end{lemm}
\begin{proof}
	The content of this lemma follows immediately from \cite{SGA1}, Expos\'e VIII, Th\'eor\`eme 5.2.
\end{proof}

\begin{coro}\label{l68}
	Let $k$ be a field; $k'$ a finite \emph{Galois} extension of $k$. Let $X$, $Y$ be geometrically integral schemes over $k$ such that there exists a $k'$-isomorphism $f:X\times_{k}k'\simto Y\times_{k}k'$. 
	For $k^{\dagger}\in \{k,k'\}$, $Z\in\{X,Y\}$, write $\Aut_{k^{\dagger}}(Z\times_{k}k')$ for the group of $k^{\dagger}$-automorphisms of $Z\times_{k}k'$. For $Z\in\{X,Y\}$, the action on the integral closure of $k$ in the ring of global sections induce an exact sequence
	\begin{displaymath}
		1\to \Aut_{k'}(Z\times_{k}k') \to \Aut_{k}(Z\times_{k}k') \to \gal(k'/k)\to 1,
	\end{displaymath}
	and the fiber product by $Z$ over $k$ on the left induces a section $s_{Z}:\gal(k'/k)\to \Aut_{k}(Z\times_{k}k')$. Denote by $F$ the group isomorphism $\Aut_{k}(X\times_{k}k')\to \Aut_{k}(Y\times_{k}k')$ naturally induced by conjugation by $f$, i.e., $f\circ t=F(t)\circ f$ for any $t\in \Aut_{k}(X\times_{k}k')$. Then there exists a $k$-isomorphism $g:X\to Y$ if and only if the sections $F\circ s_{X}$, $s_{Y}$ are conjugate by an element in $\Aut_{k}(Y\times_{k}k')$ (or equivalently, $F^{-1}\circ s_{Y}$, $s_{X}$ are conjugate by an element in $\Aut_{k}(X\times_{k}k')$). 
\end{coro}
\begin{proof}
	By \cref{l67}, there exists a $k$-isomorphism $h:X\simto Y$ if and only if there exists a $k'$-isomorphism $f':X\times_{k}k'\simto Y\times_{k}k'$ such that for any element $\alpha\in\gal(k'/k)$, $(Y\times_{k}\alpha)\circ f'=f'\circ (X\times_{k}\alpha)$. On the other hand, this condition is equivalent to the condition that for any element $\alpha\in\gal(k'/k)$, 
	\begin{displaymath}
		f\circ (f')^{-1}\circ (s_{Y}(\alpha))\circ f'\circ f^{-1}=F(s_{X}(\alpha)),
	\end{displaymath}
	i.e., $F\circ s_{X}$ and $s_{Y}$ are conjugate by $f\circ (f')^{-1}\in\Aut_{k}(Y\times_{k}k')$.
\end{proof}


Next, we observe that in the case of a smooth, geometrically connected curve of type $(0,3)$ over either a mixed-characteristic local field or a number field, one may reconstruct the function field  of the curve from the virtual fundamental group associated to a Galois section of the curve.


\begin{prop}\label{l49}
	In the situation of \cref{l48}, assume further that $k$ is a number field or a mixed-characteristic local field. Then the following hold:
	
	(i) The exact sequence
	\begin{displaymath}
	1\rightarrow \Delta_{[\pr^{2/1}_{X},s]}\rightarrow \Pi_{[\pr^{2/1}_{X},s]}\rightarrow G_{k}\rightarrow 1
	\end{displaymath}
	may be reconstructed from the abstract profinite group $\Pi_{[\pr^{2/1}_{X},s]}$.
	
	(ii) The type $(g,r)$ may be reconstructed from the abstract profinite group $\Pi_{[\pr^{2/1}_{X},s]}$, in a fashion that is functorial with respect to isomorphisms of profinite groups.
	
	(iii) $\cusp([\pr^{2/1}_{X},s])$ may be reconstructed from the abstract profinite group $\Pi_{[\pr^{2/1}_{X},s]}$, in a fashion that is functorial with respect to isomorphisms of profinite groups.
	
	(iv) Suppose that $(g,r)=(0,3)$. Then $X$ is isomorphic to $\mathbb{P}^{1}_{k}\setminus\{0,1,\infty\}$ if and only if the action of $G_{k}$ (or equivalently, $\Pi_{[\pr^{2/1}_{X},s]}$) on $\cusp([\pr^{2/1}_{X},s])$ by conjugation is trivial. 
	
	(v) Suppose that $X$ satisfies the equivalent conditions in (iv). Then elements in $\sxx$ induce natural automorphisms of the virtual variety $[\pr^{2/1}_{X},s]$, hence (by (iii)) determine a morphism 
	\begin{displaymath}
		\sxx\to \Aut_{\set}(\cusp([\pr^{2/1}_{X},s]))\simeq \mathbb{S}_{4},
	\end{displaymath}
	which is an isomorphism that maps $\vx$ to the intersection of all Sylow-$2$ subgroups in $\mathbb{S}_{4}$.
	
	(vi) Suppose that $X$ satisfies the equivalent conditions in (iv). Then one may reconstruct from the abstract profinite group $\Pi_{[\pr^{2/1}_{X},s]}$
	\begin{enumerate}
		\item[\it{(1)}] four fields, each of which is isomorphic to the function field $K(X)$ of $X$;
		\item[\it{(2)}]  an isomorphism between any two of the four fields
	\end{enumerate}
	such that the isomorphisms of (2) are closed under composition, in a fashion that is functorial with respect to isomorphisms of profinite groups.
	
\end{prop}
\begin{proof}
	Assertions (i), (ii), and (iii) follow immediately from \cref{l34}(i), \cref{l39}(ii), \cref{l38}, \cref{l25}, and \cref{l37}.
	To show assertion (iv), note that $\diag_{X} \in \cusp([\pr^{2/1}_{X},s])$ is fixed by the conjugation action of $G_{k}$. Therefore, by \cref{l30}(ii) and \cref{l48}, it suffices to show that $X$ is isomorphic to $\mathbb{P}^{1}_{k}\setminus\{0,1,\infty\}$ if and only if the action of $G_{k}$ on $\cusp([\pr^{1}_{X},s])$ by conjugation is trivial. By \cref{l48}, it suffices to show that $X$ is isomorphic to $\mathbb{P}^{1}_{k}\setminus\{0,1,\infty\}$ if and only if the action of $G_{k}$ on the $\cusp(X)$ by conjugation is trivial. This follows from \cref{l47}(iv) and \cref{l68}.
	
	Assertion (v) follows immediately from \cref{l39}(ii), (iii); \cref{l47}(v), (vi). 
	
	Next, we consider assertion (vi). Write $\{a_1,a_2,a_3,a_4\}$ for the 
	four elements of  $\cusp([\pr^{2/1}_{X},s])$. For each $i\in\{1,2,3,4\}$, write $H_{i}$ for the subgroup of $\Delta_{[\pr^{2/1}_{X},s]}$ generated by the cuspidal inertia subgroups in $a_i$ (so $H_{i}$ is normal
	in $\Pi_{[\pr^{2/1}_{X},s]}$); $\Pi_{i}$ for the quotient $\Pi_{[\pr^{2/1}_{X},s]}/H_{i}$; $\Delta_{i}$ for the kernel of the natural surjection $\Pi_{i}\surjto G_{k}$. For distinct $i,j\in\{1,2,3,4\}$,  denote the image of $H_{j}$ in $\Pi_{i}$ by $H^{i}_{j}$.
	We may assume without loss of generality that $a_{1}=\diag_{X}$. Therefore, by \cref{l48}, $\Pi_{1}$ is isomorphic to $\Pi_{X}$. We claim the following:
	
	($\star$) For distinct $i,j\in\{1,2,3,4\}$, denote by $\sigma$ the unique element in the Klein group (i.e., the intersection of all Sylow-$2$ subgroups of $\Aut(\{1,2,3,4\})\simeq \mathbb{S}_{4}$) such that $\sigma(i)=(j)$. Then, up to composition with an inner automorphism, there exists a unique isomorphism $\theta_{j}^{i}:\Pi_{i}\rightarrow \Pi_{j}$ such that $\theta_{j}^{i}$ maps $H^{i}_{m}$ onto $H^{j}_{\sigma(m)}$, and $\theta_{j}^{i}$ commutes with the natural surjections to $G_{k}$ up to composition with an inner automorphism.
	
	First, we show existence. By (v), one may take $\theta_{j}^{i}$ to be  the isomorphism induced by the natural action of $\sigma$ (considered as an element of $\vx$ [cf. \cref{l47}(vi)]) on $\Pi_{[\pr^{2/1}_{X},s]}$.
	
	
	Uniqueness follows from the following claim, which follows immediately from \cref{l72}:
	
	
	($\star\star$) The image of the natural injection
	\begin{displaymath}
	\sx\injto \Out(\Pi_{X})
	\end{displaymath}
	coincides with the set of elements that commute with the natural surjection $\Pi_{X}\surjto G_{k}$ up to composition with an inner automorphism.
	
	Now we have already reconstructed groups $\Pi_{i}$ and morphisms $\theta_{j}^{i}$ up to the composition with an inner automorphism. By \cref{l72}, this data allows one to reconstruct four function fields that are isomorphic to $K(X)$ and isomorphisms between them. Finally, it follows from the condition that $\theta_{j}^{i}$ maps $H^{i}_{m}$ onto $H^{j}_{\sigma(m)}$ that these isomorphisms satisfy the desired composition relations.
\end{proof}

Finally, we consider the situation discussed in \cref{l49} in the 
case where $g=0$ and the action of the Galois group of the base field 
on the set of cusps is trivial.

\begin{defi}\label{l100}
	Let $k$ be a field of characteristic zero such that there exists a prime number $l$ such that $k$ is $l$-cyclotomically full; $X$ a smooth hyperbolic over $k$ of type $(g,r)$; $s:G_{k}\rightarrow \Pi_{X}$ a section of the natural surjection $\Pi_{X}\surjto G_{k}$. Thus, we write $\pr^{2/1}_{X}:X_{2}\rightarrow X$ for the first projection. Then:
	
	(i) For any subset $S\subset \cusp([\pr^{2/1}_{X},s])$, write $\Pi_{[\pr^{2/1}_{X},s]}^{S}$ for the quotient of $\Pi_{[\pr^{2/1}_{X},s]}$ by the closed normal subgroup topologically generated by the elements of the conjugacy classes in $\cusp([\pr^{2/1}_{X},s])$ that are \emph{not} contained in $S$; $\Delta_{[\pr^{2/1}_{X},s]}^{S}$ for the kernel of the natural surjection $\Pi_{[\pr^{2/1}_{X},s]}^{S}\surjto G_{k}$.
	
	(ii) For any nonnegative integer $n$, denote by $\cusp^{n}([\pr^{2/1}_{X},s])$ the set of subsets of $\cusp([\pr^{2/1}_{X},s])$ of cardinality $n$.
	
	(iii) For any $\alpha\in \cusp([\pr^{2/1}_{X},s])$ and any $S\subset\cusp([\pr^{2/1}_{X},s])$, denote by $H^{S}_{[\pr^{2/1}_{X},s]}(\alpha)$ the image in $\Pi^{S}_{[\pr^{2/1}_{X},s]}$ of the closed normal subgroup of $\Pi_{[\pr^{2/1}_{X},s]}$ topologically generated by the elements of $\alpha$.

\end{defi}

\begin{lemm}\label{l69}
	Let $X$, $Y$ be smooth hyperbolic curves over a field $k$ of characteristic zero such that there exists a prime number $l$ such that $k$ is $l$-cyclotomically full. Let $f:X\to Y$ be an open immersion obtained by removing a single element $\in Y(k)$; $s:G_{k}\to \Pi_{X}$ a section. Write $s'\coloneqq f_{*}\circ s:G_{k}\to \Pi_{Y}$ for the induced section;  $u$ for the projection $X\times_{k} Y\to X$; $v$ for the projection $Y_{2}\times_{Y}X\to X$, where the fiber product $Y_{2}\times_{Y}X$ is taken relative to $\pr^{2/1}_{Y}$; $h$ for the open immersion $X_{2}\to Y_{2}\times_{Y}X$; $i$ for the open immersion $X\times_{k}X\to Y\times_{k}X$; $m$ for the open immersion $Y_{2}\times_{Y}X\to Y_{2}$; $j$ for the open immersion $X\times_{k}Y\to Y\times_{k}Y$.
	
	(i) $[u,s]$, $[v,s]$ are pointed virtual curves.

	(ii) $(i,\id,\id):[\pr^{1}_{X},s]\to [u,s]$ and $(h,\id,\id):[\pr^{2/1}_{X},s]\to [v,s]$ are degree $1$ decuspidalizations [cf. \cref{l66}(ii)].

	(iii)  The morphisms $j$ and $m$ induce isomorphisms between the exact sequences
	\begin{displaymath}
		1\to \Delta_{[u,s]}\to \Pi_{[u,s]}\to G_{k}\to 1,
	\end{displaymath}
	\begin{displaymath}
		1\to \Delta_{[v,s]}\to \Pi_{[v,s]}\to G_{k}\to 1
	\end{displaymath}
	and the exact sequences
	\begin{displaymath}
		1\to \Delta_{[\pr^{1}_{Y},s']}\to \Pi_{[\pr^{1}_{Y},s']}\to G_{k}\to 1,
	\end{displaymath}
	\begin{displaymath}
		1\to \Delta_{[\pr^{2/1}_{Y},s']}\to \Pi_{[\pr^{2/1}_{Y},s']}\to G_{k}\to 1.
	\end{displaymath}

\end{lemm}
\begin{proof}
	Assertions (i) and (ii) follow immediately from the various definitions involved. Assertion (iii) follows immediately from \cref{l01}.

	
\end{proof}

\begin{coro}\label{l70}
	In the situation of \cref{l69}, write $\alpha\in \cusp([\pr^{2/1}_{X},s])$ for the conjugacy class determined by $(h,\id,\id)$.
	Then the exact sequences
	\begin{displaymath}
	1\rightarrow \Delta_{[\pr^{2/1}_{Y},s']}\rightarrow \Pi_{[\pr^{2/1}_{Y},s']}\rightarrow G_{k}\rightarrow 1
	\end{displaymath}
	\begin{displaymath}
	1\rightarrow \Delta_{[\pr^{2/1}_{X},s]}^{\cusp([\pr^{2/1}_{X},s])\setminus\{\alpha\}}\rightarrow \Pi_{[\pr^{2/1}_{X},s]}^{\cusp([\pr^{2/1}_{X},s])\setminus\{\alpha\}}\rightarrow G_{k}\rightarrow 1
	\end{displaymath}
	[cf. \cref{l100}(i)] are naturally isomorphic.
\end{coro}
\begin{proof}
	\cref{l70} follows immediately from \cref{l29}(ii) and the portion of  \cref{l69}(ii)(iii) concerning $[v,s]$.
\end{proof}

\begin{coro}\label{l71}
	In the situation of \cref{l48}, suppose that $X$ is of type $(g,r)$. 
	Then:
	
	(i) There exists a natural bijection
	\begin{displaymath}
		\cusp([\pr^{2/1}_{X},s])\setminus \{\diag_{X}\}\simeq \cusp(X)
	\end{displaymath}
	of sets that is compatible with the respective natural $G_{k}$-actions.

	(ii) Assume that $r\geq 1$ and $2g-2+r\geq 2$. Then for every $\alpha\in \cusp([\pr^{2/1}_{X},s])\setminus \{\diag_{X}\}$,
	after possibly passing to a finite field extension of $k$, there exists an open immersion $f:X\to Y$ obtained by removing a single element $\in Y(k)$ such that $\alpha$ is determined by $f$ (where we regard $\alpha$ as an element of $\cusp(X)$ via the bijection of (1)). Moreover, $Y$ is of type $(g,r-1)$. Write $s'\coloneqq f_{*}\circ s: G_{k}\to \Pi_{Y}$. Then the exact sequences
		\begin{displaymath}
			1\rightarrow \Delta_{[\pr^{2/1}_{Y},s']}\rightarrow \Pi_{[\pr^{2/1}_{Y},s']}\rightarrow G_{k}\rightarrow 1
		\end{displaymath}
		\begin{displaymath}
			1\rightarrow \Delta_{[\pr^{2/1}_{X},s]}^{\cusp([\pr^{2/1}_{X},s])\setminus\{\alpha\}}\rightarrow \Pi_{[\pr^{2/1}_{X},s]}^{\cusp([\pr^{2/1}_{X},s])\setminus\{\alpha\}}\rightarrow G_{k}\rightarrow 1
		\end{displaymath}
		[cf. \cref{l100}(i)] are naturally isomorphic.

\end{coro}
\begin{proof}
	
	Assertion (i) follows immediately from the final portion of  \cref{l48}.
	Next, we consider assertion (ii).  Observe that $\alpha$ determines an open immersion $f:X\to Y$ such that $f$, together with $s$, $s'$, satisfy the conditions of \cref{l69}. Therefore, assertion (ii) follows immediately from \cref{l70}.
	
\end{proof}

\begin{lemm}\label{l73}
	Let $k$ be a mixed-characteristic local field; $\mathcal{C}:(X,Y,f,k,\bar{k},t,s)$ pointed virtual curve of type $(g,r)$ where $g\geq 2$. Denote by $\Delta_{\mathcal{C}}^{\emptyset}$ the quotient of $\Delta_{\mathcal{C}}$ by the closed normal subgroup topologically generated by the elements of the conjugacy classes in $\cusp(\mathcal{C})$ (where we observe that this definition is consistent with Definition 6.9(i), whenever Definition 6.9(i) is applicable). Then
	\begin{displaymath}
		\mu \coloneqq \hom(H^{2}(\Delta_{\mathcal{C}}^{\emptyset},\Zhat),\Zhat)\
	\end{displaymath}
	is a $\hat{\Z}$-module with $G_{k}$-action whose underlying topological group is isomorphic to $\hat{\Z}$. Moreover, there exists an open subgroup $H$ of $G_{k}$ such that $H$ acts on $\mu$ via the cyclotomic character.
\end{lemm}
\begin{proof}
	Without loss of generality, we may replace $k$ by a finite field extension of $k$. In particular, we may assume that $Y$ admits a $k$-rational point $y$. 
	By Corollary 3.7 and Theorem 4.20(ii)(iii), $\Delta_{\mathcal{C}}^{\emptyset}$ may be naturally identified with $\Delta_{X_{y}}$. Therefore, \cref{l73} follows immediately from [\cite{MZK9}, Proposition 1.2(ii)].


\end{proof}

\begin{theo}\label{l52}
	In the situation of \cref{l48} [cf. also \cref{l25}], assume further that $k$ is either a number field or a mixed-characteristic local field, and that $X$ is of type $(0,r)$ where $r\geq 3$.
	
	(i) The action of $G_{k}$ on $\cusp([\pr^{2/1}_{X},s])$ by conjugation is trivial if and only if every cusp of $X$ is $k$-rational.

	
	(ii) Suppose that the equivalent conditions of (i) hold. Then one may reconstruct from the abstract profinite group $\Pi_{[\pr^{2/1}_{X},s]}$, in a fashion that is functorial with respect to isomorphisms of profinite groups, an untangled [cf. \cite{MZK5}, \S0] connected graph consisting of $\frac{(r+1)!}{3!(r-2)!}$ vertices such that
	
	\begin{itemize}
		\item each vertex is equipped with a field that is isomorphic to the function field $K(X)$ of $X$;
		
		\item each edge is equipped with an isomorphism between the fields assigned to the vertices to which the edge abuts.
	\end{itemize}

	
	(iii) Suppose that the equivalent conditions of (i) hold. Then one may reconstruct from the abstract profinite group $\Pi_{[\pr^{2/1}_{X},s]}$, in a fashion that is functorial with respect to isomorphisms of profinite groups, a field $k_{\Pi}$ such that $k_{\Pi}$ is isomorphic to $k$. Moreover, for each field $K_{v}$  constructed in (ii) that corresponds to a vertex $v$, one may reconstruct, in a fashion that is functorial with respect to isomorphisms of profinite groups, an inclusion of fields $k_{\Pi}\injto K_{v}$. These inclusions are compatible with the isomorphisms associated to edges constructed in (ii).
\end{theo}
\begin{proof}
	Since $\diag_{X}$ is fixed by the conjugation action of $\Pi_{[\pr^{2/1}_{X},s]}$, assertion (i) is equivalent to the following claim [cf. \cref{l39}(ii), (iii)]:
	
	($*$) Every cusp of $X$ is $k$-rational if and only if the action of $G_{k}$ on $\cusp(X)$ is trivial.
	
	On the other hand, ($*$) follows immediately from
	[\cite{MZK3}, Lemma 4.5(v)].
	
	To verify assertion (ii), we consider the family of quotient groups $\{\Pi_{[\pr^{2/1}_{X},s]}^{S}\}$ of $\Pi_{[\pr^{2/1}_{X},s]}$, where $S$ ranges over the elements of $\cusp^{3}([\pr^{2/1}_{X},s])$ [cf. \cref{l100}(ii)]. For any distinct elements $S,S'\in \cusp^{3}([\pr^{2/1}_{X},s])$ such that $\card(S\cap S')=2$, we claim that there exists a unique (up to an inner automorphism determined by an element of $\Delta_{[\pr^{2/1}_{X},s]}^{S}$) isomorphism $\theta^{S}_{S'}:\Pi_{[\pr^{2/1}_{X},s]}^{S}\simto \Pi_{[\pr^{2/1}_{X},s]}^{S'}$ that satisfies the following condition:
	
	($**$): For any $\alpha\in \cusp([\pr^{2/1}_{X},s])$, $\theta^{S}_{S'}$ maps $H^{S}_{[\pr^{2/1}_{X},s]}(\alpha)$ [cf. Definition 6.9(iii)] isomorphically onto $H^{S'}_{[\pr^{2/1}_{X},s]}(\sigma(\alpha))$, where $\sigma$ is the unique element in $\Aut_{\set}(\cusp([\pr^{2/1}_{X},s]))$ that $\sigma$ permutes the two elements in $S\cap S'$, permutes the two elements in $(S\cup S')\setminus(S\cap S')$, and fixes the other elements. Moreover, $\theta^{S}_{S'}$ is compatible with the natural surjections $\Pi_{[\pr^{2/1}_{X},s]}^{S}\surjto G_{k}$, $\Pi_{[\pr^{2/1}_{X},s]}^{S'}\surjto G_{k}$.
	
	Uniqueness follows immediately from [\cite{MZK8}, Theorem A]. To verify existence, we consider the following two cases
	
	\begin{enumerate}
		\item[(a)] $\diag_{X} \notin S\cup S'$;
		\item[(b)] $\diag_{X} \in S\cup S'$.
	\end{enumerate}
	
	To verify case (a), we note that, by \cref{l48}, both $\Pi_{[\pr^{2/1}_{X},s]}^{S}$ and $\Pi_{[\pr^{2/1}_{X},s]}^{S'}$ may be regarded as quotients (by some normal subgroup generated by elements in conjugacy classes of cuspidal inertia groups) of $\Pi_{[\pr^{1}_{X},s]}$. Therefore, case (a) follows immediately from \cref{l48} and \cref{l47}(vi).

	
	To verify case (b), by \cref{l70}, we may assume that $X$ is of type $(0,3)$. Therefore, case (b) follows immediately from \cref{l49}(vi). This completes the proof of existence. Now assertion (ii) follows immediately from Proposition 6.2, by taking the {\it graph} in the statement of assertion (ii) to be the graph whose {\it vertices} are the elements of $\cusp^{3}([\pr^{2/1}_{X},s])$, and whose {\it edges} are the pairs $(S,S')$ of distinct elements $S, S'\in \cusp^{3}([\pr^{2/1}_{X},s])$ such that $\card(S\cap S')=2$.
	
	
	
	Assertion (iii) follows immediately from assertion (ii) and Proposition 6.2.
	Here, we also apply the {\it compatibility} stated in the final portion of ($**$), together with Lemma 6.13 and the {\it  ``functoriality/compatibility''} with respect to finite \etale coverings discussed in [[16], Remark 1.10.1, (i)(ii)], to construct a cyclotome ``$\mu$'' as in Lemma 6.13 that is {\it independent} of the vertex $v$ and compatible with the isomorphisms $\theta^S_{S'}$ of ($**$).

	
\end{proof}

\begin{theo}\label{l74}
	Let $k$ (respectively, $k'$) be either a number field or a mixed-characteristic local field; $X$ (respectively, $X'$) a smooth, geometrically connected curve over $k$ (respectively, $k'$) of type $(0,r)$ (respectively, $(0,r')$) where $r\geq 3$ (respectively, $r'\geq 3$) ; $s:G_{k}\rightarrow \Pi_{X}$ (respectively, $s':G_{k'}\rightarrow \Pi_{X'}$) a section. Suppose that there exists an isomorphism
	\begin{displaymath}
	f:\Pi_{[\pr^{2/1}_{X},s]}\xrightarrow{\sim} \Pi_{[\pr^{2/1}_{X'},s']}.
	\end{displaymath}
	Then:
	
	(i) $f(\Delta_{[\pr^{2/1}_{X},s]}) =\Delta_{[\pr^{2/1}_{X'},s']}$. In particular, $f$ induces an isomorphism between $G_{k}$ and $G_{k'}$.
	
	(ii) $r=r'$.
	
	(iii) The isomorphism $f$ induces a bijection $\cusp([\pr^{2/1}_{X},s])\simto \cusp([\pr^{2/1}_{X'},s'])$ that is compatible with the actions of $G_{k}$ and $G_{k'}$.
	
	(iv) The isomorphism $f$ induces an isomorphism $k\simto k'$ of fields. 
	
	(v) There exists an isomorphism $K(X)\simto K(X')$ that is compatible with the field isomorphism constructed in (iv).
	
	(vi) If $r=r'=3$, then there exists an isomorphism $X\simto X'$ of curves that is compatible with the isomorphism $k\simto k'$ of (iv).
\end{theo}
\begin{proof}
	Assertions (i), (ii), and (iii) follow immediately from \cref{l39}(i)(ii), Proposition 5.2, and \cref{l37}.
	
	Next, we consider assertion (iv). Let $H$ (respectively, $H'$) be the maximal open subgroup of $G_{k}$ (respectively, $G_{k'}$) that acts trivially on $\cusp([\pr^{2/1}_{X},s])$ (respectively, $\cusp([\pr^{2/1}_{X'},s'])$). By (iii), $f$ induces an isomorphism between $H$ and $H'$. Let $k_{H}$ (respectively, $k'_{H'}$) be the field determined by $H$ (respectively, $H'$). By \cref{l52}(iii), $f$ induces a natural isomorphism $k_{H}\simto k'_{H'}$ that is compatible with the respective natural actions of $G_{k}/H$ and $G_{k'}/H'$ on $k_{H}$ and $k'_{H'}$, relative to the isomorphism $G_{k}/H\simto G_{k'}/H'$ induced by $f$ [cf. assertion (i)]. By Galois theory, we have natural isomorphisms $k \simto (k_{H})^{G_{k}/H}$, $k' \simto (k'_{H'})^{G_{k'}/H'}$. Therefore, $f$ induces a natural isomorphism $k\simto k'$ of fields. 
	
	Next, we consider assertion (v). Denote by $E$ (respectively, $E'$) the subset of $\cusp([\pr^{2/1}_{X},s])$ (respectively, $\cusp([\pr^{2/1}_{X'},s'])$) that consists of elements that are \emph{not} fixed by the action of $G_{k}$ (respectively, $G_{k'}$). Thus, the bijection of (iii) induces a bijection between $E$ and $E'$. One verifies immediately that there are three cases to consider:
	
	\begin{enumerate}
		\item $\card(E)=0$;
		\item $\card(E)=2$;
		\item $\card(E)\geq 3$.
	\end{enumerate}
	
	Case ($1$) follows immediately from \cref{l52}(ii)(iii). To verify case ($2$), we begin by observing that it follows 
    from \cref{l48} that both $X$ and $X'$ admit a $k$-rational point, hence that $K(X)$ (respectively, $K(X')$) is 
    isomorphic over $k$ (respectively, $k'$) to the function field of $\mathbf{P}^{1}_{k}$ (respectively, 
    $\mathbf{P}^{1}_{k'}$). Thus, case ($2$) follows immediately from assertion (iv).  To verify case ($3$), we 
    observe that $f$ induces an isomorphism $\Pi^{E}_{[\pr^{2/1}_{X},s]}\simto \Pi^{E'}_{[\pr^{2/1}_{X'},s']}$, where 
    we note that both the domain and the codomain of this isomorphism is a quotient by cuspidal inertia subgroups 
    of the quotient of ``$\Pi_{[\pr^{{2/1}}_X,s]}$'' by ``$H$'' constructed in \cref{l48}. Therefore, case ($3$) follows 
    immediately from \cref{l48} and \cref{l72}.
	
	Next, we consider assertion (vi). One maintain the notation involving $E$, $E'$, as well as the classification into 
    three cases, of the proof of (v).	
	Case ($1$) follows immediately from \cref{l49}(iv).
	To verify case ($2$), we begin by observing that, in light of our assumption that $r=r'=3$, it follows that, by 
    applying a suitable automorphism in the group ``$J$'' of \cref{l47}(vi), together with the descent criterion of 
    \cref{l67}, we may assume without loss of generality that the bijection of assertion (iii) maps $\diag_X 
    \mapsto \diag_{X'}$, where we apply the notation of \cref{l48}.  Thus, $f$ induces an isomorphism between the 
    respective quotients ``$\Pi_{[\pr^{2/1}_{X},s]}/H$'' constructed in \cref{l48}, so case ($2$) follows immediately 
    from \cref{l48} and \cref{l72}.  
	To verify case ($3$), we observe that $f$ induces an isomorphism 
    $\Pi^{E}_{[\pr^{2/1}_{X},s]}\simto \Pi^{E'}_{[\pr^{2/1}_{X'},s']}$ between suitable quotients by cuspidal inertia 
    groups of the respective quotients ``$\Pi_{[\pr^{2/1}_{X},s]}/H$'' constructed in \cref{l48}. Therefore, case ($3$) 
    follows immediately from \cref{l48} and \cref{l72}.

\end{proof}


\bibliographystyle{plain}
\bibliography{arithmetic_lib.bib}

\begin{thebibliography}{10}

\bibitem{SGA3}
M.~Artin, J.-E. Bertin, M.~Demazure, A.~Grothendieck, P.~Gabriel, M.~Raynaud, and J.-P. Serre.
\newblock {\em Sch\'emas en groupes (SGA 3)}.
\newblock S\'eminaire de G\'eom\'etrie Alg\'ebrique de l'Institut des Hautes \'Etudes Scientifiques. Institut des Hautes \'Etudes Scientifiques, 1963/1966.

\bibitem{FC}
C.-L. Chai and G.~Faltings.
\newblock {\em Degeneration of Abelian Varieties}.
\newblock Ergebnisse der Mathematik und ihrer Grenzgebiete (3), 22. Springer-Verlag, 1990.

\bibitem{Milne}
G.~Cornell and J.~H. Silverman.
\newblock {\em Arithmetic Geometry}.
\newblock Springer-Verlag, 1986.

\bibitem{SGA1}
A.~Grothendieck.
\newblock {\em Rev\^etements \'Etales et Groupe Fondamental (SGA 1)}.
\newblock Lecture Notes in Mathematics, 224. Springer-Verlag, 1971.

\bibitem{Hos}
Y.~Hoshi.
\newblock {The Exactness of the Log Homotopy Sequence}.
\newblock {\em {Hiroshima Math. J.}}, 39(1):61--122, 2009.

\bibitem{Hos3}
Y.~Hoshi.
\newblock {The Grothendieck Conjecture for Hyperbolic Polycurves of Lower Dimension}.
\newblock {\em {J. Math. Sci. Univ. Tokyo}}, 21(2):153--219, 2014.

\bibitem{Hos4}
Y.~Hoshi.
\newblock {The Geometry of Hyperbolic Curvoids}.
\newblock {\em {Publ. Res. Inst. Math. Sci. }}, 59(2):1--55, 2023.

\bibitem{nodnon}
Y.~Hoshi and S.~Mochizuki.
\newblock {On the Combinatorial Anabelian Geometry of Nodally Nondegenerate Outer Representations}.
\newblock {\em {Hiroshima Math. J.}}, 41:275--342, 2011.

\bibitem{Kat1}
K.~Kato.
\newblock {Logarithmic Structure of Fontaine-Illusie}.
\newblock In {\em Alg. Analysis, Geom. and Number Th.} 191-224, Johns Hopkins Univ. Press, 1989.

\bibitem{MZK4}
S.~Mochizuki.
\newblock {Extending Families of Curves over Log Regular Schemes}.
\newblock {\em {J. reine angew. Math.}}, 511:43--71, 1999.

\bibitem{MZK8}
S.~Mochizuki.
\newblock {The Local Pro-$p$ Anabelian Geometry of Curves}.
\newblock {\em {Invent. Math.}}, 138:319--423, 1999.

\bibitem{MZK6}
S.~Mochizuki.
\newblock {The Absolute Anabelian Geometry of Hyperbolic Curves}.
\newblock In {\em {Galois Theory and Modular Forms}}. {77-122. Kluwer Acad. Publ.}, 2004.

\bibitem{MZK9}
S.~Mochizuki.
\newblock {Absolute Anabelian Cuspidalizations of Proper Hyperbolic Curves}.
\newblock {\em {J. Math. Kyoto Univ.}}, 47:451--539, 2007.

\bibitem{MZK3}
S.~Mochizuki.
\newblock {Topics in Absolute Anabelian Geometry I: Generalities}.
\newblock {\em {J. Math. Sci. Univ. Tokyo}}, 19(2):139--242, 2012.

\bibitem{MZK5}
S.~Mochizuki.
\newblock {Topics in Absolute Anabelian Geometry II: Decomposition Groups and Endomorphisms}.
\newblock {\em {J. Math. Sci. Univ. Tokyo}}, 20(2):171--269, 2013.

\bibitem{MZK2}
S.~Mochizuki.
\newblock {Topics in Absolute Anabelian Geometry III: Global Reconstruction Algorithms}.
\newblock {\em {J. Math. Sci. Univ. Tokyo}}, 22(4):939--1156, 2015.

\bibitem{MZK1}
S.~Mochizuki and A.~Tamagawa.
\newblock {The Algebraic and Anabelian Geometry of Configuration Spaces}.
\newblock {\em {Hokkaido Math. J.}}, 37(1):75--131, 2008.

\bibitem{Mumf}
D.~Mumford.
\newblock {\em {Abelian Varieties}}.
\newblock Oxford Univ. Press, 1974.

\bibitem{Neu}
J.~Neukirch.
\newblock {\em {Algebraic Number Theory}}.
\newblock {Grundlehren der Mathematischen Wissenschaften, 322}. Springer-Verlag, 1999.

\bibitem{NSW}
J{\"u}rgen Neukirch, Alexander Schmidt, and Kay Wingberg.
\newblock {\em Cohomology of Number Fields}, volume 323 of {\em Grundlehren der mathematischen Wissenschaften}.
\newblock Springer, Berlin, 2nd edition, 2013.

\end{thebibliography}

\end{document}